\documentclass[11pt,a4paper,leqno]{amsart}
\usepackage[english]{babel}

\usepackage[latin1]{inputenc}
\usepackage[T1]{fontenc}
\usepackage{amsfonts}
\usepackage{amsmath}
\usepackage{amssymb}
\usepackage{eurosym}
\usepackage{mathrsfs}
\usepackage{palatino}
\usepackage{color}
\usepackage{esint}
\usepackage{url}
\usepackage{verbatim}
\allowdisplaybreaks[4]
\usepackage{enumerate}

\newcommand{\R}{\mathbb{R}}
\newcommand{\C}{\mathbb{C}}

\newcommand{\Z}{\mathbb{Z}}
\newcommand{\E}{\mathbb{E}}

\newcommand{\calF}{\mathcal{F}}

\newcommand{\calJ}{\mathcal{J}}

\newcommand{\bla}{\big \langle}
\newcommand{\bra}{\big \rangle}

\numberwithin{equation}{section}

\newcommand{\ud}[0]{\,\mathrm{d}}

\newcommand{\esssup}[0]{\operatornamewithlimits{ess\,sup}}

\newcommand{\BMO}[0]{\operatorname{BMO}}
\newcommand{\bmo}[0]{\operatorname{bmo}}

\newcommand{\ch}[0]{\operatorname{ch}}

\newcommand{\calD}[0]{\mathcal{D}}

\newcommand{\wt}[1]{{\widetilde{#1}}}

\swapnumbers
\theoremstyle{plain}
\newtheorem{thm}[equation]{Theorem}
\newtheorem{lem}[equation]{Lemma}
\newtheorem{prop}[equation]{Proposition}

\theoremstyle{definition}

\theoremstyle{remark}
\newtheorem{rem}[equation]{Remark}

\author{Kangwei Li}
\address[K.L.]{BCAM (Basque Center for Applied Mathematics), Alameda de Mazarredo 14, 48009 Bilbao, Spain}
\email{kli@bcamath.org}

\author{Henri Martikainen}
\address[H.M.]{Department of Mathematics and Statistics, University of Helsinki, P.O.B. 68, FI-00014 University of Helsinki, Finland}
\email{henri.martikainen@helsinki.fi}

\author{Emil Vuorinen}
\address[E.V.]{Department of Mathematics and Statistics, University of Helsinki, P.O.B. 68, FI-00014 University of Helsinki, Finland}
\email{emil.vuorinen@helsinki.fi}

\title[Bi-parameter Bloom type inequality for iterated commutators]{Bloom type inequality for bi-parameter singular integrals: efficient proof and iterated commutators}

\makeatletter
\@namedef{subjclassname@2010}{%
  \textup{2010} Mathematics Subject Classification}
\makeatother

\subjclass[2010]{42B20}
\keywords{Representation theorems, iterated commutators, Bloom's inequality}

\begin{document}

\begin{abstract}
Utilising some recent ideas from our bilinear bi-parameter theory, we give an efficient proof of a two-weight Bloom type inequality
for iterated commutators of linear bi-parameter singular integrals. We prove that if $T$ is a bi-parameter singular integral
satisfying the assumptions of the bi-parameter representation theorem, then
$$
\| [b_k,\cdots[b_2, [b_1, T]]\cdots]\|_{L^p(\mu) \to L^p(\lambda)} \lesssim_{[\mu]_{A_p}, [\lambda]_{A_p}} \prod_{i=1}^k\|b_i\|_{\bmo(\nu^{\theta_i})} ,
$$
where $p \in (1,\infty)$, $\theta_i \in [0,1]$, $\sum_{i=1}^k\theta_i=1$, $\mu, \lambda \in A_p$, $\nu := \mu^{1/p}\lambda^{-1/p}$. Here
$A_p$ stands for the bi-parameter weights in $\R^n \times \R^m$ and $\bmo(\nu)$ is a suitable weighted little BMO space.
We also simplify the proof of the known first order case.
\end{abstract}

\maketitle

\section{Introduction}
We recently developed in \cite{LMV} a lot of theory for general bilinear bi-parameter singular integrals using modern dyadic analysis -- in particular,
we proved various bilinear bi-parameter commutator estimates. This lead us to discover an improved general principle for approaching bi-parameter commutator estimates
of dyadic model operators. In this paper we use our method to give an efficient
proof of Bloom type inequalities for iterated commutators of bi-parameter singular integrals.  Our objective is to offer a proof with a very transparent structure.
The iterated result is new in the bi-parameter setting, and its proof benefits greatly from this structure. Our proof of the first order case is short.

With a Bloom type inequality we understand the following. Given some operator $A^b$, the definition of which depends naturally on some function $b$,
we seek for a two-weight estimate
$$
\|A^b\|_{L^p(\mu) \to L^p(\lambda)} \lesssim \|b\|_{\BMO(\nu)},
$$
where $p \in (1,\infty)$, $\mu, \lambda \in A_p$, $\nu := \mu^{1/p}\lambda^{-1/p}$, and $\BMO(\nu)$ is some suitable weighted $\BMO$ space.
Usually $A^b$ is some commutator, like $[b,T]f := bTf - T(bf)$, where $T$ is a singular integral operator. Bloom \cite{Bl} achieved such an inequality
for $T = H$ -- the Hilbert transform. Holmes--Lacey--Wick \cite{HLW, HLW2} gave a modern proof and generalised Bloom's result
to the case of a general (one-parameter) Calder\'on--Zygmund operator. The iterated case is by Holmes--Wick \cite{HW} (see also
Hyt\"onen \cite{Hy4} for a proof via the Cauchy integral trick). An improved iterated case is by Lerner--Ombrosi--Rivera-R\'ios \cite{LOR2}: in \cite{HW, Hy4} there
is some single $b \in \BMO \cap \BMO(\nu)$, while in \cite{LOR2} the iteration is taken using
$b \in \BMO(\nu^{1/k}) \supset \BMO \cap \BMO(\nu)$ (see also the related paper \cite{GHST} by Garc\'ia--Cuerva, Harboure, Segovia and Torrea).
In \cite{LOR2} it is said that it seems that their bound
cannot be obtained by a simple inductive argument.
Some multilinear (one-parameter) Bloom type inequalities are considered by Kunwar--Ou \cite{KO}. Commutator estimates are in general very important
and widely studied -- for some other very recent references see e.g. Hyt\"onen \cite{Hy5} and Lerner--Ombrosi--Rivera-R\'ios \cite{LOR}.

A model of a bi-parameter singular integral operator in $\R^n \times \R^m$ is $T_1 \otimes T_2$, where
$T_1$ and $T_2$ are usual singular integrals in $\R^n$ and $\R^m$, respectively.
The general definition of a bi-parameter singular integral $T$ requires that 
$\langle Tf_1, f_2\rangle$, $f_i = f_i^1 \otimes f_i^2$, can be written using different kernel representations depending on whether
\begin{enumerate}
\item $\operatorname{spt} f_1^1 \cap \operatorname{spt}f_2^1 = \emptyset$ and $\operatorname{spt} f_{1}^2 \cap \operatorname{spt} f_{2}^2 = \emptyset$,
\item $\operatorname{spt} f_1^1 \cap \operatorname{spt}f_2^1 = \emptyset$ or
\item $\operatorname{spt} f_{1}^2 \cap \operatorname{spt}f_{2}^2 = \emptyset$.
\end{enumerate}
In the first case we have a so-called full kernel representation, while in cases $2$ and $3$
a partial kernel representations holds in $\R^n$ or $\R^m$, respectively.
The bi-parameter representation theorem \cite{Ma1} by one of us has enabled the development of deep commutator
estimates also in the bi-parameter setting. The representation holds under natural $T1$ conditions involving
the product $\BMO$ space of Chang and Fefferman \cite{CF1} and some weak testing conditions. It allows
to reduce the commutator estimates of singular integrals
to those of model operators $U$, where $U$ is a so-called bi-parameter shift, partial paraproduct or full paraproduct.
We will only need these model operators in this paper and they are recalled in Section \ref{sec:firstorder}.
As the somewhat lengthy kernel estimates and testing conditions of $T$ are not explicitly needed here, we refer to \cite{Ma1} for the remaining details.

Using the dyadic representation theorem Ou, Petermichl and Strouse proved in \cite{OPS} that $[b,T] \colon L^2(\R^{n+m}) \to L^2(\R^{n+m})$, when $T$ is a paraproduct free bi-parameter singular integral and $b$ is a little BMO function. This is the important base case for more complicated multi-parameter commutator estimates involving product BMO and iterated commutators of the form $[T_1, [b, T_2]]$ -- see again \cite{OPS} and Dalenc--Ou \cite{DaO}.
For the earlier deep commutator lower bounds in the Hilbert and Riesz settings see Ferguson--Lacey \cite{FL} and Lacey--Petermichl--Pipher--Wick \cite{LPPW}. The paper \cite{OPS} was eventually generalised to concern all bi-parameter singular integrals satisfying $T1$ conditions by Holmes--Petermichl--Wick \cite{HPW}. In fact, \cite{HPW} proves much more: Bloom's inequality in the bi-parameter setting.
The multi-parameter commutator scene is again very active, see also e.g. Duong--Li--Ou--Pipher--Wick \cite{DLOPW}, which
is a very recent paper concerning commutators of multi-parameter flag singular integrals.

In \cite{LMV} we explain that the presence of non-cancellative Haar functions $h_I^0$ in many of the bi-parameter model operators seem to have caused
a lot of technical troubles in previous bi-parameter commutator estimates.
Our guideline is to expand $bf$ using bi-parameter martingales in $\langle bf, h_{I} \otimes h_J\rangle$, using one-parameter martingales in
$\langle bf, h_{I}^0 \otimes h_J\rangle$ (or $\langle bf, h_I \otimes h_J^0\rangle$), and not to expand at all in $\langle bf, h_{I}^0 \otimes h_J^0\rangle$.
Moreover, when a non-cancellative Haar function appears a suitable average of $b$ is added and subtracted.
See Section \ref{sec:marprod} for the general details and e.g. \eqref{eq:partial} for an example of the resulting simple decomposition.
In \cite{HPW}
everything was always reduced to a so called remainder term, which essentially entails expanding $bf$ in the bi-parameter sense
in all of the above situations. However, this remainder term has a particularly nice structure only when there are no non-cancellative Haar functions (the shift case) -- otherwise
it can lead to some difficult tail terms.

In this paper we want to use the above decomposition idea from \cite{LMV} and showcase how it simplifies things in the linear bi-parameter setting. The Bloom setting is demanding, but the proof framework adapts nicely even to this generality. Our treatment of first order commutators is very different in many ways compared to \cite{HPW} -- that is, the simplifications in the decomposition itself, which are described above, are not the only difference -- we also estimate differently. We exploit the known one-weight boundedness of the model operators even more: most terms
arising from our new decomposition can be estimated directly by combining the weighted boundedness of the model operators and some Bloom type estimates of appropriate auxiliary operators, such as,
$$
f \mapsto \sup_R \frac {1_R}{|R|}\int_R |b-\langle b \rangle_R||f|,
$$
where $R = I \times J$ is a dyadic rectangle. The bounds for this maximal function presented in Proposition \ref{prop:bloomforMb}
rely on an interesting estimate of Fefferman \cite{Fe3} concerning the maximal function $f \mapsto \sup_R 1_R  \langle |f| \rangle_R^{\lambda}$,
$\langle f \rangle_R^{\lambda} := \lambda(R)^{-1} \int_R f\ud \lambda$, defined using an $A_{\infty}$-weight $\lambda$
(see also Appendix \ref{app2}).
Modern sparse domination methods are also useful in some parts
of the proof -- we use such estimates from \cite{LOR} and \cite{LMOV, LMV}. For example, a certain special term $U^b$ associated to a model operator $U$ and the commuting function $b$ needs to be estimated directly. The estimate \eqref{eq:ApplySparse} that follows from
sparse domination techniques is very effective for this.

We also can, for the first time, prove a Bloom type inequality for iterated commutators of bi-parameter singular integrals.
Our main theorem is:
\begin{thm}\label{thm:main}
Let $T$ be a bi-parameter singular integral satisfying the $T1$ type assumptions of the dyadic representation theorem \cite{Ma1}.
Let also $p \in (1,\infty)$, $\mu, \lambda \in A_p$ and $\nu := \mu^{1/p}\lambda^{-1/p}$. Then we have
$$
\| [b, T]\|_{L^p(\mu) \to L^p(\lambda)} \lesssim_{[\mu]_{A_p}, [\lambda]_{A_p}} \|b\|_{\bmo(\nu)}
$$
and, more generally,
$$
\| [b_k,\cdots[b_2, [b_1, T]]\cdots]\|_{L^p(\mu) \to L^p(\lambda)} \lesssim_{[\mu]_{A_p}, [\lambda]_{A_p}} \prod_{i=1}^k\|b_i\|_{\bmo(\nu^{\theta_i})},
$$
where $\theta_i \in [0,1]$ and $\sum_{i=1}^k \theta_i=1$. Here
$A_p$ stands for the bi-parameter weights in $\R^n \times \R^m$ and $\bmo(\nu)$ is a suitable weighted little BMO space.
\end{thm}
\begin{rem}
In the proof of the iterated commutator estimate, we will only prove the second order case, since the proof structure is such that it is clear
how to continue the iteration.

Notice also that choosing $b_1 = \cdots = b_k = b$ and $\theta_1 = \cdots = \theta_k = 1/k$ we get a bi-parameter analog of \cite{LOR2}, while choosing $\theta_1 = 1$ (and the rest zero) we get analogs of \cite{HW, Hy4}. However, the first is the better choice as $\bmo(\nu^{1/k}) \supset \bmo \cap \bmo(\nu)$. Indeed, similarly as in the one-parameter case
\cite{LOR2}, this is seen by using that $\langle \nu \rangle_{R}^{\theta} \lesssim_{[\nu]_{A_2}} \langle \nu^{\theta} \rangle_{R}$ for all $\theta \in (0,1)$ and rectangles $R$ (this estimate follows from Theorem 2.1 in \cite{CN} by iteration).
\end{rem}
We also mention that some experts may find Appendix \ref{app1} interesting: it proves that little BMO is contained in the product BMO -- even in the weighted situation --
using only relatively elementary tools. We confess that we were only previously aware of a proof of this in the unweighted situation, and that proof depended
on the deep commutator result of Ferguson--Lacey \cite{FL}. This weighted result is mentioned in \cite{HPW} without proof.

\section{Definitions and preliminaries}
\subsection{Basic notation}
We denote $A \lesssim B$ if $A \le CB$ for some constant $C$ that can depend on the dimension of the underlying spaces, on integration exponents, and on various other constants appearing in the assumptions. We denote $A \sim B$ if $B \lesssim A \lesssim B$.

We work in the bi-parameter setting in the product space $\R^{n+m}$.
In such a context $x = (x_1, x_2)$ with $x_1 \in \R^n$ and $x_2 \in \R^m$.
We often take integral pairings with respect to one of the two variables only:
If $f \colon \R^{n+m} \to \C$ and $h \colon \R^n \to \C$, then $\langle f, h \rangle_1 \colon \R^{m} \to \C$ is defined by
$$
\langle f, h \rangle_1(x_2) = \int_{\R^n} f(y_1, x_2)h(y_1)\ud y_1.
$$

\subsection{Dyadic notation, Haar functions and martingale differences}
We denote a dyadic grid in $\R^n$ by $\calD^n$ and a dyadic grid in $\R^m$ by $\calD^m$. If $I \in \calD^n$, then $I^{(k)}$ denotes the unique dyadic cube $S \in \calD^n$ so that $I \subset S$ and $\ell(S) = 2^k\ell(I)$. Here $\ell(I)$ stands for side length. Also, $\text{ch}(I)$ denotes the dyadic children of $I$, i.e., $I' \in \ch(I)$ if $I' \in \calD^n$, $I' \subset I$ and $\ell(I') = \ell(I)/2$. We sometimes write
$\calD = \calD^n \times \calD^m$.

When $I \in \calD^n$ we denote by $h_I$ a cancellative $L^2$ normalised Haar function. This means the following.
Writing $I = I_1 \times \cdots \times I_n$ we can define the Haar function $h_I^{\eta}$, $\eta = (\eta_1, \ldots, \eta_n) \in \{0,1\}^n$, by setting
\begin{displaymath}
h_I^{\eta} = h_{I_1}^{\eta_1} \otimes \cdots \otimes h_{I_n}^{\eta_n}, 
\end{displaymath}
where $h_{I_i}^0 = |I_i|^{-1/2}1_{I_i}$ and $h_{I_i}^1 = |I_i|^{-1/2}(1_{I_{i, l}} - 1_{I_{i, r}})$ for every $i = 1, \ldots, n$. Here $I_{i,l}$ and $I_{i,r}$ are the left and right
halves of the interval $I_i$ respectively. The reader should carefully notice that $h_I^0$ is the non-cancellative Haar function for us and that
in some other papers a different convention is used.
If $\eta \in \{0,1\}^n \setminus \{0\}$ the Haar function is cancellative: $\int h_I^{\eta} = 0$. We usually suppress the presence of $\eta$
and simply write $h_I$ for some $h_I^{\eta}$, $\eta \in \{0,1\}^n \setminus \{0\}$. Then $h_Ih_I$ can stand for $h_I^{\eta_1} h_I^{\eta_2}$, but we always treat
such a product as a non-cancellative function (which it is in the worst case scenario $\eta_1 = \eta_2$).

For $I \in \calD^n$ and a locally integrable function $f\colon \R^n \to \C$, we define the martingale difference
$$
\Delta_I f = \sum_{I' \in \textup{ch}(I)} \big[ \bla f \bra_{I'} -  \bla f \bra_{I} \big] 1_{I'}.
$$
Here $\bla f \bra_I = \frac{1}{|I|} \int_I f$. We also write $E_I f = \bla f \bra_I 1_I$.
Now, we have $\Delta_I f = \sum_{\eta \ne 0} \langle f, h_{I}^{\eta}\rangle h_{I}^{\eta}$, or suppressing the $\eta$ summation, $\Delta_I f = \langle f, h_I \rangle h_I$, where $\langle f, h_I \rangle = \int f h_I$. A martingale block is defined by
$$
\Delta_{K,i} f = \mathop{\sum_{I \in \calD^n}}_{I^{(i)} = K} \Delta_I f, \qquad K \in \calD^n.
$$

Next, we define bi-parameter martingale differences. Let $f \colon \R^n \times \R^m \to \C$ be locally integrable.
Let $I \in \calD^n$ and $J \in \calD^m$. We define the martingale difference
$$
\Delta_I^1 f \colon \R^{n+m} \to \C, \Delta_I^1 f(x) := \Delta_I (f(\cdot, x_2))(x_1).
$$
Define $\Delta_J^2f$ analogously, and also define $E_I^1$ and $E_J^2$ similarly.
We set
$$
\Delta_{I \times J} f \colon \R^{n+m} \to \C, \Delta_{I \times J} f(x) = \Delta_I^1(\Delta_J^2 f)(x) = \Delta_J^2 ( \Delta_I^1 f)(x).
$$
Notice that $\Delta^1_I f = h_I \otimes \langle f , h_I \rangle_1$, $\Delta^2_J f = \langle f, h_J \rangle_2 \otimes h_J$ and
$ \Delta_{I \times J} f = \langle f, h_I \otimes h_J\rangle h_I \otimes h_J$ (suppressing the finite $\eta$ summations).
Martingale blocks are defined in the natural way
$$
\Delta_{K \times V}^{i, j} f  =  \sum_{I\colon I^{(i)} = K} \sum_{J\colon J^{(j)} = V} \Delta_{I \times J} f = \Delta_{K,i}^1( \Delta_{V,j}^2 f) = \Delta_{V,j}^2 ( \Delta_{K,i}^1 f).
$$

\subsection{Weights}
A weight $w(x_1, x_2)$ (i.e. a locally integrable a.e. positive function) belongs to bi-parameter $A_p(\R^n \times \R^m)$, $1 < p < \infty$, if
$$
[w]_{A_p(\R^n \times \R^m)} := \sup_{R} \bla w \bra_R \bla w^{1-p'} \bra_R^{p-1} < \infty,
$$
where the supremum is taken over $R = I \times J$, where $I \subset \R^n$ and $J \subset \R^m$ are cubes
with sides parallel to the axes (we simply call such $R$ rectangles). Here $1/p + 1/p' = 1$, i.e., $p'$ is the dual exponent of $p$.
We have
$$
[w]_{A_p(\R^n\times \R^m)} < \infty \textup { iff } \max\big( \esssup_{x_1 \in \R^n} \,[w(x_1, \cdot)]_{A_p(\R^m)}, \esssup_{x_2 \in \R^m}\, [w(\cdot, x_2)]_{A_p(\R^n)} \big) < \infty,
$$
and that $\max\big( \esssup_{x_1 \in \R^n} \,[w(x_1, \cdot)]_{A_p(\R^m)}, \esssup_{x_2 \in \R^m}\, [w(\cdot, x_2)]_{A_p(\R^n)} \big) \le [w]_{A_p(\R^n\times \R^m)}$, while
the constant $[w]_{A_p}$ is dominated by the maximum to some power.
Of course, $A_p(\R^n)$ is defined similarly as $A_p(\R^n \times \R^m)$ -- just take the supremum over cubes $Q$. For the basic theory
of bi-parameter weights consult e.g. \cite{HPW}.

Also, recall that $w \in A_{\infty}(\R^n)$ if
$$
[w]_{A_{\infty}(\R^n)} = \sup_Q \Big( \frac{1}{|Q|} \int_Q w \Big) \operatorname{exp}\Big( \frac{1}{|Q|} \int_Q \log w^{-1} \Big) < \infty,
$$
where the supremum is taken over all the cubes $Q \subset \R^n$. We will use that $A_p \subset A_{\infty}$, and also some estimates that are valid for $A_{\infty}$ weights. 

\subsection{Maximal functions and standard estimates}
Given $f \colon \R^{n+m} \to \C$ and $g \colon \R^n \to \C$ we denote the dyadic maximal functions
by
$$
M_{\calD^n}g(x) := \sup_{I \in \calD^n} \frac{1_I(x)}{|I|}\int_I |g(y)| \ud y
$$
and
$$
M_{\calD^n, \calD^m} f(x_1, x_2) := \sup_{R \in \calD^n \times \calD^m}  \frac{1_R(x_1, x_2)}{|R|}\iint_R |f(y_1, y_2)|\ud y_1 \ud y_2.
$$
We also set $M^1_{\calD^n} f(x_1, x_2) =  M_{\calD^n}(f(\cdot, x_2))(x_1)$. The operator $M^2_{\calD^m}$ is defined similarly.
We record the following standard estimates, which are used repeatedly below. 
\begin{lem}\label{lem:standardEst}
For $p \in (1,\infty)$ and $w \in A_p(\R^n \times \R^m)$ the weighted square function estimates
\begin{align*}
\| f \|_{L^p(w)}
& \sim_{[w]_{A_p(\R^n \times \R^m)}} \Big\| \Big( \mathop{\sum_{I \in \calD^n}}_{J \in \calD^m} |\Delta_{I \times J} f|^2 \Big)^{1/2} \Big\|_{L^p(w)} \\
&\sim_{[w]_{A_p(\R^n \times \R^m)}}  \Big\| \Big(  \sum_{I \in \calD^n} |\Delta_I^1 f|^2 \Big)^{1/2} \Big\|_{L^p(w)}
\sim_{[w]_{A_p(\R^n \times \R^m)}} \Big\| \Big(  \sum_{J \in \calD^m} |\Delta_J^2 f|^2 \Big)^{1/2} \Big\|_{L^p(w)}
\end{align*}
hold. Moreover, for $p, s \in (1,\infty)$ we have the Fefferman--Stein inequality
$$
\Big\| \Big( \sum_j |M f_j |^s \Big)^{1/s} \Big\|_{L^p(w)} \le C([w]_{A_p}) \Big\| \Big( \sum_{j} | f_j |^s \Big)^{1/s} \Big\|_{L^p(w)}.
$$
Here $M$ can e.g. be $M_{\calD^n}$, $M_{\calD^n}^1$ or $M_{\calD^n, \calD^m}$. Finally, we have
$$
\| \varphi_{\calD^n, \calD^m}^1 f\|_{L^p(w)} 
\sim_{[w]_{A_p}}
\Big\| \Big( \sum_{I \in \calD^n} \frac{1_I}{|I|} \otimes [M_{\calD^m} \langle f, h_I \rangle_1]^2 \Big)^{1/2} \Big\|_{L^p(w)} 
\lesssim_{[w]_{A_p}} \|f\|_{L^p(w)},
$$
where
$$
\varphi_{\calD^n, \calD^m}^1 f  := \sum_{I \in \calD^n} h_I \otimes M_{\calD^m} \langle f, h_I \rangle_1.
$$
The function $\varphi_{\calD^n, \calD^m}^2 f$ is defined in the symmetric way and satisfies the same estimates.
\end{lem}
One easy way to show such estimates is to reduce to $p=2$ via standard extrapolation. When $p=2$ it is especially easy to
use one-parameter results iteratively. See e.g. \cite{CMP, CWW} for one-parameter square function results and their history.

If an average is with respect to a different measure than the Lebesgue measure we can e.g. write
$\langle f \rangle_R^{\lambda} := \frac{1}{\lambda(R)} \int_R f\ud \lambda$, and similarly
we can write $M_{\calD^n, \calD^m, \lambda}f = \sup_R 1_R  \langle |f| \rangle_R^{\lambda}$.

\subsection{BMO spaces}\label{ss:bmo}
Given $w \in A_2(\R^n)$ we say that a locally integrable function $b \colon \R^n \to \C$ belongs to the weighted dyadic BMO space $\BMO_{\calD^n}(w)$ if
$$
\|b\|_{\BMO_{\calD^n}(w)} := \sup_{I \in \calD^n} \frac{1}{w(I)} \int_I |b - \langle b \rangle_I| < \infty.
$$
The space $\BMO(w)$ can be defined using the norm defined by the supremum over all dyadic grids of
the above dyadic norms.

Given $w \in A_2(\R^n \times \R^m)$ we say that a locally integrable function $b \colon \R^{n+m} \to \C$ belongs to the weighted dyadic little BMO space $\bmo_{\calD^n, \calD^m}(w)$ if
$$
\|b\|_{\bmo_{\calD^n, \calD^m}(w)} := \sup_{R \in \calD^n \times \calD^m} \frac{1}{w(R)} \int_R |b - \langle b \rangle_R| < \infty.
$$
Again, the space $\bmo(w)$ is defined via the supremum of the dyadic norms. We have
$$
\|b\|_{\bmo_{\calD^n, \calD^m}(w)} \sim \max\big( \esssup_{x_1 \in \R^n} \, \|b(x_1, \cdot)\|_{\BMO_{\calD^m}(w(x_1, \cdot))}, \esssup_{x_2 \in \R^m}\, \|b(\cdot, x_2)\|_{\BMO_{\calD^n}(w(\cdot, x_2))} \big).
$$
Moreover, we have the two-weight John--Nirenberg property
\begin{equation}\label{eq:2JN}
\|b\|_{\bmo_{\calD^n, \calD^m}(\nu)} \sim_{[\mu]_{A_p}, [\lambda]_{A_p}} \sup_{R \in \calD^n \times \calD^m} \Big( \frac{1}{\mu(R)} \int_{R} |b - \langle b \rangle_{R}|^p \lambda \Big)^{1/p},
\end{equation}
if $p \in (1, \infty)$, $\mu, \lambda \in A_p$ and $\nu := \mu^{1/p}\lambda^{-1/p}$. Notice that here $\nu \in A_2$. For these see \cite{HPW}.

Finally, we have the product BMO space. Given $w \in A_2(\R^n \times \R^m)$ set
$$
\|b\|_{\BMO_{\textup{prod}}^{\calD^n, \calD^m}(w)} := 
\sup_{\Omega} \Big( \frac{1}{w(\Omega)} \mathop{\sum_{I \in \calD^n, J \in \calD^m}}_{I \times J \subset \Omega} |\langle b, h_I \otimes h_J\rangle|^2 \langle w \rangle_{I \times J}^{-1} \Big)^{1/2},
$$
where the supremum is taken over those sets $\Omega \subset \R^{n+m}$ such that $|\Omega| < \infty$ and such that for every $x \in \Omega$ there exist
$I \in \calD^n, J \in \calD^m$ so that $x \in I \times J \subset \Omega$.
The non-dyadic product BMO space can be defined using the norm defined by the supremum over all dyadic grids of
the above dyadic norms. 

It is stated in \cite{HPW} (without proof or reference) that $\bmo(w) \subset \BMO_{\textup{prod}}(w)$, $w \in A_2$. This embedding $\|b\|_{\BMO_{\textup{prod}}^{\calD^n, \calD^m}(w)} \lesssim_{[w]_{A_2}} \|b\|_{\bmo_{\calD^n, \calD^m}(w)}$ is used in the main proof only via the fact that it implies that
\eqref{eq:WeightedH1BMO} also holds for $\bmo(w)$ functions. We give a proof of this result in Appendix \ref{app1}.

\subsection{Commutators}
We briefly discuss one way to understand how the commutators are defined, and how all the pairings and expansions appearing in our proof can be seen to be well defined.
For example, we discuss the second order case. Let $b_i \in \bmo(\nu^{\theta_i})$, $i=1,2$, be given, where $\nu = \mu^{1/p}\lambda^{-1/p}$, $\mu, \lambda \in A_p$, $p \in (1,\infty)$.
Define
$$
\calF = \calF(b_1, b_2) = \bigcup_{k=1}^{\infty} \{f \colon \R^{n+m} \to \R \colon \operatorname{spt}\, f \subset B(0,k) \cap \{|b_1|, |b_2| \le k\} \textup{ and } |f| \le k\}.
$$
For $f_1, f_2 \in \calF$ the pairing $\langle [b_2, [b_1, T]]f_1, f_2 \rangle$ is well defined (if $T$ is e.g. a singular integral satisfying the assumptions of the representation theorem)
and $\calF$ is dense in $L^p(\mu)$ and $L^{p'}(\lambda^{1-p'})$.
Moreover, for some $k$ we can replace $b_i$ by $b_{i,k} = \max(\min(b_i, k), -k)$. Notice that
$\| b_{i,k} \|_{\bmo(\nu^{\theta_i})} \lesssim \| b_i \|_{\bmo(\nu^{\theta_i})}$. This can be seen by using identities like
$\max(c,d) = (c+d+|c-d|)/2$, and showing that $h \in \bmo(\nu)$ implies $|h| \in \bmo(\nu)$. 

These considerations imply that below we may assume that the little BMO functions
$b_1, b_2$ are bounded and $f_1, f_2$ are bounded and compactly supported, which makes everything legitimate.

\section{Martingale difference expansions of products}\label{sec:marprod}
We recall from \cite{LMV} our modified strategy of expanding commutators.
A product $bf$ paired with Haar functions is expanded in the bi-parameter fashion only if both of the Haar functions are cancellative. In a mixed
situation we expand only in $\R^n$ or $\R^m$, and in the remaining fully non-cancellative situation we do not expand at all.
Our protocol also entails the following: when pairing with a non-cancellative Haar function we add and subtract a suitable average of $b$.

Let $\calD^n$ and $\calD^m$ be some fixed dyadic grids in $\R^n$ and $\R^m$, respectively, and write $\calD= \calD^n \times \calD^m$.
In what follows we sum over $I \in \calD^n$ and $J \in \calD^m$.

\subsubsection*{Paraproduct operators}
We define certain standard paraproduct operators:
\begin{align*}
A_1(b,f) &= \sum_{I, J} \Delta_{I \times J} b \Delta_{I \times J} f, \,\,
A_2(b,f) = \sum_{I, J} \Delta_{I \times J} b E_I^1\Delta_J^2 f, \\
A_3(b,f) &= \sum_{I, J} \Delta_{I \times J} b \Delta_I^1 E_J^2  f, \,\,
A_4(b,f) = \sum_{I, J} \Delta_{I \times J} b \bla f \bra_{I \times J},
\end{align*}
and
\begin{align*}
A_5(b,f) &= \sum_{I, J} E_I^1 \Delta_J^2 b \Delta_{I \times J} f, \,\,
A_6(b,f) = \sum_{I, J}  E_I^1 \Delta_J^2 b  \Delta_I^1 E_J^2  f, \\
A_7(b,f) &= \sum_{I, J} \Delta_I^1 E_J^2  b \Delta_{I \times J} f, \,\,
A_8(b,f) = \sum_{I, J}  \Delta_I^1 E_J^2 b E_I^1 \Delta_J^2 f.
\end{align*}
The operators are grouped into two collections, since they are handled differently (using product BMO or little BMO estimates, respectively).

We also define
$$
a^1_1(b,f) = \sum_I \Delta_I^1 b \Delta_I^1 f \qquad \textup{and} \qquad
a^1_2(b,f) = \sum_I \Delta_I^1 b E_I^1 f.
$$
The operators $a^2_1(b,f)$ and $a^2_2(b,f)$ are defined analogously.
\begin{lem}\label{lem:basicAa}
Let $\pi_b$ be $A_i(b,\cdot)$, $i=1,\ldots, 8$, or $a_j^1(b,\cdot)$, $a_j^2(b,\cdot)$, $j=1,2$. Suppose $b\in \bmo(\nu)$, where $\nu=\mu^{\frac 1p}\lambda^{-\frac 1p}$, $\mu, \lambda\in A_p$ and $p \in (1,\infty)$. Then
\begin{align*}
\|\pi_b\|_{L^p(\mu)\rightarrow L^p(\lambda)} \lesssim_{[\mu]_{A_p}, [\lambda]_{A_p}} \|b\|_{\bmo(\nu)}.
\end{align*}
\end{lem}
\begin{rem}
Notice that $\nu = (\lambda^{1-p'})^{1/p'}(\mu^{1-p'})^{-1/p'}$, where $\lambda^{1-p'}, \mu^{1-p'} \in A_{p'}$, so that
the natural dual statement concerning $\|\pi_b\|_{L^{p'}(\lambda^{1-p'}) \to L^{p'}(\mu^{1-p'}) }$ follows.
\end{rem}
\begin{proof}[Proof of Lemma \ref{lem:basicAa}]
The operators $A_i(b, \cdot)$ (but in a somewhat different form) are already discussed in \cite{HPW}.
To aid the reader we note that the proofs essentially write themselves if one knows certain weighted $H^1$-BMO type duality estimates.
For $i=1,\ldots, 4$ we use
\begin{equation}\label{eq:WeightedH1BMO}
\sum_{I, J} |\langle b, h_I \otimes h_J \rangle| |A_{IJ}| \lesssim_{[\nu]_{A_2}} \|b\|_{\BMO_{\textup{prod}}^{\calD^n, \calD^m}(\nu)} \Big\|\Big( \sum_{I,J} |A_{IJ}|^2 \frac{1_{I \times J}}{|I \times J|} \Big)^{1/2} \Big\|_{L^1(\nu)}
\end{equation}
and the fact that $\bmo_{\calD^n, \calD^m}(\nu) \subset  \BMO_{\textup{prod}}^{\calD^n, \calD^m}(\nu)$. A proof of \eqref{eq:WeightedH1BMO} is recorded in \cite{HPW} (but
even this weighted version was well-known according to them).
For $i = 5, \ldots, 8$ we may use
the one parameter analog of the estimate \eqref{eq:WeightedH1BMO} in various ways -- e.g. through the fact that it implies
(as $[\langle \nu \rangle_{I,1}]_{A_2} \le [\nu]_{A_2}$ and $\|\langle b \rangle_{I,1}\|_{\BMO(\langle \nu \rangle_{I,1})} \le  \|b\|_{\bmo(\nu)}$) that
$$
\sum_J \Big| \Big\langle b, \frac{1_I}{|I|} \otimes h_J\Big\rangle\Big| |A_{IJ}| \lesssim_{[\nu]_{A_2}} \|b\|_{\bmo_{\calD^n, \calD^m}(\nu)}
\Big\|\Big( \sum_{J} |A_{IJ}|^2 \frac{1_J}{|J|} \Big)^{1/2} \Big\|_{L^1(\langle \nu \rangle_{I,1})}.
$$
Of course, the operators $a_j^1(b,\cdot)$, $a_j^2(b,\cdot)$, $j=1,2$, can also be handled with the one parameter analog of \eqref{eq:WeightedH1BMO}.
\end{proof}

For $I_0 \in \calD^n$ and $ J_0 \in \calD^m$ we will now introduce our expansions of $\langle bf, h_{I_0} \otimes h_{J_0}\rangle$, $\big \langle bf, h_{I_0} \otimes \frac{1_{J_0}}{|J_0|}\big\rangle$ and $\langle bf \rangle_{I_0 \times J_0}$.
\subsubsection*{Expansion of $\langle bf, h_{I_0} \times h_{J_0} \rangle$}
There holds
$$
1_{I_0 \times J_0} b
= \sum_{\substack{I_1\times J_1 \in \calD \\ I_1 \times J_1 \subset I_0 \times J_0}}\Delta_{I_1 \times J_1} b
+\sum_{\substack{J_1 \in \calD^m \\ J_1 \subset J_0}} E^1_{I_0} \Delta^2_{J_1} b
+ \sum_{\substack{I_1 \in \calD^n \\ I_1 \subset I_0}} \Delta^1_{I_1} E^2_{J_0} b
+ E_{I_0 \times J_0} b.
$$
Let us denote these terms by $I_j$, $j=1,2,3,4$, in the respective order.
We have the corresponding decomposition of $f$, whose terms we denote by $II_i$, $i=1,2,3,4$. Calculating carefully the pairings
$\langle I_j II_i, h_{I_0} \otimes h_{J_0} \rangle$ we see that
\begin{equation}\label{eq:biparEX}
\langle bf, h_{I_0} \otimes h_{J_0} \rangle = \sum_{i=1}^8 \langle A_i(b, f), h_{I_0} \otimes h_{J_0} \rangle + \langle b \rangle_{I_0 \times J_0} \langle f, h_{I_0} \otimes h_{J_0} \rangle.
\end{equation}
\subsubsection*{Expansion of $\big \langle bf, h_{I_0} \otimes \frac{1_{J_0}}{|J_0|}\big\rangle$}
This time we write
$
1_{I_0} b  = \sum_{\substack{I_1 \in \calD^n \\ I_1 \subset I_0}}\Delta_{I_1}^1 b + E_{I_0}^1 b,
$
and similarly for $f$. Calculating $\langle bf, h_{I_0} \rangle_1$ we see that
\begin{equation}\label{eq:1EX}
\begin{split}
\Big \langle bf, h_{I_0} \otimes \frac{1_{J_0}}{|J_0|}\Big\rangle &= \sum_{i=1}^2 \Big\langle a_i^1(b,f), h_{I_0} \otimes \frac{1_{J_0}}{|J_0|} \Big\rangle \\
&+ \bla (\langle b \rangle_{I_0,1} - \langle b \rangle_{I_0 \times J_0}) \langle f, h_{I_0}\rangle_1\bra_{J_0}
+ \langle b \rangle_{I_0 \times J_0} \Big \langle f, h_{I_0} \otimes \frac{1_{J_0}}{|J_0|}\Big\rangle.
\end{split}
\end{equation}

When we have $\langle bf \rangle_{I_0 \times J_0}$ we do not expand at all:
\begin{equation}\label{eq:noEX}
\langle bf \rangle_{I_0 \times J_0} = \langle (b-\langle b \rangle_{I_0 \times J_0})f \rangle_{I_0 \times J_0}
+ \langle b \rangle_{I_0 \times J_0} \langle f\rangle_{I_0 \times J_0}.
\end{equation}
All of our commutators are simply decomposed using \eqref{eq:biparEX}, \eqref{eq:1EX} (and its symmetric form) and \eqref{eq:noEX} whenever
the relevant pairings/averages appear.

\section{First order commutator}\label{sec:firstorder}
Let $U = U^{k,v} = U^{k,v}_{\calD^n, \calD^m}$, $k = (k_i)$, $v = (v_i)$, $0 \le k_i \in \Z$ and $0 \le v_i \in \Z$, $i=1,2$, be a dyadic bi-parameter operator
(defined using fixed dyadic grids $\calD^n$ and $\calD^m$)
such that
\begin{equation*}
\langle Uf_1,f_2 \rangle
= \sum_{\substack{K \in \calD^n \\ V \in \calD^m}}
\sum_{\substack{I_1, I_2 \in \calD^n \\ I_1^{(k_1)} = I_2^{(k_2)} = K}}
\sum_{\substack{J_1, J_2 \in \calD^m \\ J_1^{(v_1)} = J_2^{(v_2)}  = V}} a_{K, V, (I_i), (J_j)}
\langle f_1, \wt h_{I_1} \otimes  \wt h_{J_1}\rangle \langle f_2, \wt h_{I_2} \otimes \wt h_{J_2}\rangle ,
\end{equation*}
where $a_{K, V, (I_i), (J_j)}$ are scalars and for all $i=1,2$ we have $ \wt h_{I_i}= h_{I_i}$ (a cancellative Haar function)
for all $I_i \in \calD^n$ or $\wt h_{I_i}= 1_{I_i}/|I_i|$ for all $I_i\in \calD^n$,
and similarly with the functions $\wt h_{J_j}$. To prove a Bloom type inequality for $[b,T]$, where $T$ is a bi-parameter singular integral, it is enough to prove a Bloom type
inequality for $[b,U]$, where $U$ can be a so called bi-parameter shift, partial paraproduct or a full paraproduct (we will recall what these mean later).
This is because of the dyadic bi-parameter representation
theorem \cite{Ma1} -- one only has to be maintain a polynomial dependence of $k_1, k_2, v_1, v_2$.

The basic structure is the following.
\begin{enumerate}
\item The shift case: We have
$$\langle f_1, \wt{h}_{I_1}\otimes \wt{h}_{J_1}\rangle \langle f_2, \wt{h}_{I_2}\otimes \wt{h}_{J_2} \rangle =  \langle f_1,  h_{I_1} \otimes  h_{J_1}\rangle
\langle f_2,  h_{I_2} \otimes  h_{J_2}\rangle.$$
\item The partial paraproduct case: We have $k_1 = k_2 = 0$ and
$$\langle f_1, \wt{h}_{I_1}\otimes \wt{h}_{J_1}\rangle \langle f_2, \wt{h}_{I_2}\otimes \wt{h}_{J_2} \rangle = \Big\langle f_1,  \frac{1_K}{|K|} \otimes  h_{J_1}\Big\rangle
\langle f_2,  h_{K} \otimes  h_{J_2}\rangle$$ or the symmetric case, or we have $v_1 = v_2 = 0$ and
$$
\langle f_1, \wt{h}_{I_1}\otimes \wt{h}_{J_1}\rangle \langle f_2, \wt{h}_{I_2}\otimes \wt{h}_{J_2} \rangle = \Big\langle f_1,  h_{I_1}\otimes  \frac{1_V}{|V|}\Big\rangle
\langle f_2,  h_{I_2} \otimes  h_{V}\rangle
$$
or the symmetric case.
\item The full paraproduct case: We have $k_1 = k_2 = v_1 = v_2 = 0$ and
$$\langle f_1, \wt{h}_{I_1}\otimes \wt{h}_{J_1}\rangle \langle f_2, \wt{h}_{I_2}\otimes \wt{h}_{J_2} \rangle = \langle f_1 \rangle_{K \times V}
\langle f_2,  h_{K} \otimes  h_{V}\rangle
$$ or the symmetric case,
or we have $k_1 = k_2 = v_1 = v_2 = 0$ and
$$\langle f_1, \wt{h}_{I_1}\otimes \wt{h}_{J_1}\rangle \langle f_2, \wt{h}_{I_2}\otimes \wt{h}_{J_2} \rangle = \Big\langle f_1,  h_{K} \otimes  \frac{1_V}{|V|}\Big\rangle
\Big\langle f_2,  \frac{1_K}{|K|} \otimes  h_{V}\Big\rangle$$
or the symmetric case.
\end{enumerate}
Most terms arising from our decomposition of $[b, U]$ can in fact be handled using the fact that all the model operators satisfy
for all $1<p<\infty$ and $w\in A_p(\R^{n} \times \R^{m})$ that
\begin{equation}\label{eq:1WeightForModel}
\begin{split}
\sum_{\substack{K \in \calD^n \\ V \in \calD^m}}
\sum_{\substack{I_1, I_2 \in \calD^n \\ I_i^{(k_i)} = K}}
\sum_{\substack{J_1, J_2 \in \calD^m \\ J_j^{(v_j)} = V}} \big|a_{K, V, (I_i), (J_i)}  \langle f_1, \wt{h}_{I_1}\otimes &\wt{h}_{J_1}\rangle \langle f_2, \wt{h}_{I_2}\otimes \wt{h}_{J_2} \rangle\big| \\
&\lesssim C([w]_{A_p}) \|f_1\|_{L^p(w)}\|f_2\|_{L^{p'}(w^{1-p'})}.
\end{split}
\end{equation}
Given some suitable BMO function $b$ let us also define $U^b$ via
\begin{align*}
\langle U^b f_1,f_2 \rangle
= \sum_{\substack{K \in \calD^n \\ V \in \calD^m}}
\sum_{\substack{I_1, I_2 \in \calD^n \\ I_i^{(k_i)} = K}}
\sum_{\substack{J_1, J_2 \in \calD^m \\ J_j^{(v_j)} = V}}&a_{K, V, (I_i), (J_j)} [\langle b \rangle_{I_2 \times J_2} - \langle b \rangle_{I_1 \times J_1}] \\
&\times
\langle f_1, \wt h_{I_1} \otimes  \wt h_{J_1}\rangle \langle f_2, \wt h_{I_2} \otimes \wt h_{J_2}\rangle.
\end{align*}
In the unweighted (or one weight case) the boundedness of $U^b$ can be reduced to \eqref{eq:1WeightForModel} via the simple observation that
$$
|\langle b \rangle_{I_2 \times J_2} - \langle b \rangle_{I_1 \times J_1}| \lesssim \|b\|_{\bmo(\R^n \times \R^m)} \max(k_i, v_i).
$$
However, if we want to prove a Bloom type inequality for $U^b$, and this is key for the Bloom type inequality for $[b, U]$,
we have to run a harder adaptation of the proof of \eqref{eq:1WeightForModel}. This requires recalling more carefully what the assumptions
about the coefficients $a_{K,V, \ldots}$ are in each case.
Notice also that $U^b = 0$ when $k=v=0$ i.e. $U^b$ does not arise in the full paraproduct case. Moreover, the Bloom type inequality for $U^b$ is much harder when $U$
is a partial paraproduct compared to the case that $U$ is a shift (we use sparse bounds of bilinear paraproducts to handle the partial paraproduct case).

Despite having to deal with $U^b$ separately, it is extremely convenient to blackbox \eqref{eq:1WeightForModel}.
Such a weighted bound for all model operators was first recorded in \cite{HPW}. The
proof is essentially the same with or without weights (in the weighted case one just uses weighted versions of square function and maximal function bounds at the end).
We note that
a reader who is not familiar with the fundamental basic bound \eqref{eq:1WeightForModel} can essentially read the proof from the current paper also. Indeed, for full paraproducts one can consult Lemma \ref{lem:basicAa}, and for the other model operators the bounds proved for $U^b$ are harder,
and in fact an easier version of those arguments can also be used to get \eqref{eq:1WeightForModel}.
\subsection{The shift case}
We show that if $U = U^{k,v}$ is a shift then
$$
|\langle [b, U]f_1, f_2\rangle| \lesssim_{[\mu]_{A_p}, [\lambda]_{A_p}} \|b\|_{\bmo(\nu)} (1+\max(k_i, v_i)) \|f_1\|_{L^p(\mu)} \|f_2\|_{L^{p'}(\lambda^{1-p'})}.
$$

Using our general decomposition philosophy from Section \ref{sec:marprod} we see that
\begin{equation}\label{eq:Com1ofShift}
\begin{split}
&\langle [b, U]f_1, f_2\rangle = \sum_{i=1}^8 \langle Uf_1, A_i(b, f_2) \rangle
- \sum_{i=1}^8 \langle U( A_i(b, f_1)), f_2 \rangle + \langle U^b f_1, f_2\rangle.
\end{split}
\end{equation}
The first term is easy using Lemma \ref{lem:basicAa} and \eqref{eq:1WeightForModel} as
\begin{align*}
|\langle Uf_1, A_i(b, f_2) \rangle| &\le \|Uf_1\|_{L^p(\mu)} \|A_i(b, f_2)\|_{L^{p'}(\mu^{1-p'})} \\
&\lesssim_{[\mu]_{A_p}, [\lambda]_{A_p}} \|b\|_{\bmo(\nu)} \|f_1\|_{L^p(\mu)} \|f_2\|_{L^{p'}(\lambda^{1-p'})},
\end{align*}
and the second one is handled similarly.

To handle $U^b$ we begin by splitting
\begin{equation}\label{eq:av-av-split}
\begin{split}
\langle b \rangle_{I_2 \times J_2} - \langle b \rangle_{I_1 \times J_1} &=
[\langle b \rangle_{I_2 \times J_2} - \langle b \rangle_{K \times J_2}]
+ [\langle b \rangle_{K \times J_2} - \langle b \rangle_{K \times V}] \\
&+ [\langle b \rangle_{K \times V} -  \langle b \rangle_{K \times J_1}] +
[\langle b \rangle_{K \times J_1} - \langle b \rangle_{I_1 \times J_1}].
\end{split}
\end{equation}
The resulting four terms are essentially symmetric, so we only deal with the first one.  There holds that
\begin{equation}\label{eq:av-av}
|\langle b \rangle_{I_2 \times J_2} - \langle b \rangle_{K \times J_2}| \lesssim  \|b\|_{\bmo(\nu)} \sum_{\substack{L \in \calD^n \\ I_2 \subsetneq L \subset K}}
\frac{\nu(L \times J_2)}{|L \times J_2|}.
\end{equation}
Using this we see that it is enough to fix one $l \in \{1, \dots, k_2\}$ and estimate the term
\begin{equation}\label{eq:ShiftAv-Av}
\begin{split}
\sum_{\substack{K \in \calD^n \\ V \in \calD^m}}
&\sum_{\substack{L \in \calD^n \\ L^{(k_2-l)}=K}}
\sum_{\substack{J_2 \in \calD^m \\ J_2^{(v_2)} = V}} \\
& 
\frac{\nu(L \times J_2)}{|L \times J_2|}
\sum_{\substack{I_1, I_2 \in \calD^n \\ I_1^{(k_1)}=K \\ I_2^{(l)} = L}} \sum_{\substack{J_1 \in \calD^m \\ J_1^{(v_1)} = V}}
\big| a_{K, V, (I_i), (J_i)} \langle f_1,  h_{I_1} \otimes  h_{J_1}\rangle
\langle f_2,  h_{I_2} \otimes  h_{J_2}\rangle\big|.
\end{split}
\end{equation}

Now we use the fact that
$$
|a_{K, V, (I_i), (J_i)}| \le \frac{|I_1|^{1/2}|I_2|^{1/2}}{|K|} \frac{|J_1|^{1/2}|J_2|^{1/2}}{|V|},
$$
which implies that
\begin{align*}
\frac{\nu(L \times J_2)}{|L \times J_2|} &
\sum_{\substack{I_1, I_2 \in \calD^n \\ I_1^{(k_1)}=K \\ I_2^{(l)} = L}} \sum_{\substack{J_1 \in \calD^m \\ J_1^{(v_1)} = V}}
\big| a_{K, V, (I_i), (J_i)} \langle f_1,  h_{I_1} \otimes  h_{J_1}\rangle
\langle f_2,  h_{I_2} \otimes  h_{J_2}\rangle\big| \\
&\le \nu(L \times J_2) \langle |\Delta_{K \times V}^{k_1, v_1} f_1| \rangle_{K \times V} \langle |\Delta_{K \times V}^{k_2, v_2} f_2| \rangle_{L \times J_2} \\
&\le \iint 1_{L \times J_2} M_{\calD^n, \calD^m}(\Delta_{K \times V}^{k_1, v_1} f_1) M_{\calD^n, \calD^m}(\Delta_{K \times V}^{k_2, v_2} f_2) \nu.
\end{align*}
Using this we see that \eqref{eq:ShiftAv-Av} can be dominated by
\begin{align*}
&\sum_{\substack{K \in \calD^n \\ V \in \calD^m}}  \iint M_{\calD^n, \calD^m}(\Delta_{K \times V}^{k_1, v_1} f_1) M_{\calD^n, \calD^m}(\Delta_{K \times V}^{k_2, v_2} f_2) \nu \\
&\le \Big\| \Big( \sum_{\substack{K \in \calD^n \\ V \in \calD^m}}(M_{\calD^n, \calD^m}\Delta_{K \times V}^{k_1, v_1} f_1)^2 \Big)^{1/2} \Big\|_{L^p(\mu)}
\Big\| \Big(\sum_{\substack{K \in \calD^n \\ V \in \calD^m}} (M_{\calD^n, \calD^m}\Delta_{K \times V}^{k_2, v_2} f_2)^2 \Big)^{1/2} \Big\|_{L^{p'}(\lambda^{1-p'})} \\
&\lesssim_{[\mu]_{A_p}, [\lambda]_{A_p}} \|f_1\|_{L^p(\mu)} \|f_2\|_{L^{p'}(\lambda^{1-p'})},
\end{align*}
where in the second step we used that $\nu = \mu^{1/p}\lambda^{-1/p}$. We are done with the shifts.

\subsection{The partial paraproduct case}
We now deal with the partial paraproducts, and we choose the symmetry
\begin{equation}\label{eq:FormOfPP}
\langle Uf_1, f_2\rangle = \sum_{\substack{K \in \calD^n \\ V \in \calD^m}} \sum_{\substack{I_1, I_2 \in \calD^n \\ I_1^{(k_1)}= I_2^{(k_2)} = K}} a_{K, V, (I_i)}
\Big\langle f_1,  h_{I_1}\otimes  \frac{1_V}{|V|}\Big\rangle
\langle f_2,  h_{I_2} \otimes  h_{V}\rangle.
\end{equation}
We will show that
$$
|\langle [b, U]f_1, f_2\rangle| \lesssim_{[\mu]_{A_p}, [\lambda]_{A_p}} \|b\|_{\bmo(\nu)} (1+\max(k_1, k_2)) \|f_1\|_{L^p(\mu)} \|f_2\|_{L^{p'}(\lambda^{1-p'})}.
$$

Using our general decomposition philosophy from Section \ref{sec:marprod} we see that
\begin{equation}\label{eq:partial}
\begin{split}
\langle &[b, U]f_1, f_2\rangle = \sum_{i=1}^8 \langle Uf_1, A_i(b, f_2) \rangle 
- \sum_{i=1}^2 \langle U(a_i^1(b, f_1)), f_2 \rangle + \langle U^b f_1, f_2\rangle \\
&- \sum_{\substack{K \in \calD^n \\ V \in \calD^m}} \sum_{\substack{I_1, I_2 \in \calD^n \\ I_1^{(k_1)}= I_2^{(k_2)} = K}} a_{K, V, (I_i)}
\bla (\langle b \rangle_{I_1, 1} - \bla b \rangle_{I_1 \times V})\langle f_1, h_{I_1}\rangle_1 \bra_V\langle f_2,  h_{I_2} \otimes  h_{V}\rangle.
\end{split}
\end{equation}
The first two terms are handled precisely as in the shift case. The last term is directly under control using \eqref{eq:1WeightForModel} and
the following lemma.

\begin{lem}\label{lem:PhiNuLemma}
Let $p \in (1, \infty)$ and $\mu, \lambda \in A_p(\R^n \times \R^m)$.
Assume that $b \in \bmo(\nu)$, where $\nu=\mu^{1/p} \lambda^{-1/p}$.
Let $I \in \calD^n$ and $J \in \calD^m$.  Then
$$
\big|\bla (\langle b \rangle_{I, 1} - \langle b \rangle_{I \times J})\langle f, h_{I}\rangle_1 \bra_J \big |
\lesssim_{[\nu]_{A_2}} \| b \|_{\bmo(\nu)}\Big\langle \varphi_{\calD^n, \calD^m}^{\nu,1 } f, h_I \otimes \frac{1_J}{|J|} \Big\rangle,
$$
where
$$
\varphi_{\calD^n, \calD^m}^{\nu,1 } f
= \sum_{I \in \calD^n} h_I \otimes   M_{\calD^m}( \langle f, h_I \rangle_1) \langle \nu \rangle_{I,1}.
$$
Moreover, we have
$$
\|\varphi_{\calD^n, \calD^m}^{\nu,1 } \|_{L^p(\mu) \to L^p(\lambda)} \lesssim_{[\mu]_{A_p}, [\lambda]_{A_p}} 1.
$$
\end{lem}

\begin{proof}
We will use the one parameter estimate
\[
\frac 1{|J|}\int_J |b_0-\langle b_0 \rangle_J|  |g|\lesssim [w]_{A_\infty}  \|b_0\|_{\BMO(w)} \frac 1{|J|}\int_J M_{\calD^m}(g)w.
\]
Let us prove this. Using Lemma 5.1 in \cite{LOR} we find a sparse family $\mathcal S = \mathcal S(J, b)$ such that
\begin{equation}\label{eq:lor}
|b_0-\langle b_0\rangle_J| 1_J\le 2^{n+2} \sum_{\substack{Q\in \mathcal S\\ Q\subset J}}\big \langle |b_0-\langle b_0\rangle_Q|\big\rangle_Q 1_Q.
\end{equation}
Therefore, we have
\begin{align*}
\int_J |b_0-\langle b_0 \rangle_J|  |g|&\lesssim \sum_{\substack{Q\in \mathcal S\\ Q\subset J}}\big \langle |b_0-\langle b_0\rangle_Q|\big\rangle_Q \int_Q |g|\\
&\le \|b_0\|_{\BMO(w)} \sum_{\substack{Q\in \mathcal S\\ Q\subset J}}\big \langle |g|\big\rangle_Q w(Q)\\
&\le \|b_0\|_{\BMO(w)} \sum_{\substack{Q\in \mathcal S\\ Q\subset J}}\big [\langle M_{\calD^m}(g)^{\frac 12}\big\rangle_Q^{w}]^2 w(Q) \\
&\lesssim [w]_{A_\infty}  \|b_0\|_{\BMO(w)} \int_J M_{\calD^m}(g)w,
\end{align*}where in the last step we have used the Carleson embedding theorem (notice that Lebesgue sparse implies $w$-Carleson).

Now, we have
\begin{align*}
\big|\bla (\langle b \rangle_{I, 1} - \bla b \rangle_{I \times J})\langle f, h_{I}\rangle_1 \bra_J \big | &\lesssim [\langle \nu \rangle_{I,1}]_{A_{\infty}}
\|\langle b \rangle_{I,1}\|_{\BMO(\langle \nu \rangle_{I,1})}
\bla M_{\calD^m}(\langle f, h_I\rangle_1)\langle \nu \rangle_{I,1} \bra_J \\
&=  [\langle \nu \rangle_{I,1}]_{A_{\infty}}
\|\langle b \rangle_{I,1}\|_{\BMO(\langle \nu \rangle_{I,1})} \Big\langle \varphi_{\calD^n, \calD^m}^{\nu,1 } f, h_I \otimes \frac{1_J}{|J|} \Big\rangle,
\end{align*}
and then recall that $[\langle \nu \rangle_{I,1}]_{A_2} \le [\nu]_{A_2}$ and $\|\langle b \rangle_{I,1}\|_{\BMO(\langle \nu \rangle_{I,1})} \le  \|b\|_{\bmo(\nu)}$.

Next, we have
\begin{align*}
\|\varphi_{\calD^n, \calD^m}^{\nu,1 }f \|_{L^p(\lambda)} &\sim_{[\lambda]_{A_p}} \Big\| \Big( \sum_{I \in \calD^n} \frac{1_I}{|I|} \otimes [M_{\calD^m}( \langle f, h_I \rangle_1) \langle \nu \rangle_{I,1}]^2 \Big)^{1/2} \Big\|_{L^p(\lambda)} \\
&\lesssim_{[\lambda]_{A_p}} \Big\| \Big( \sum_{I \in \calD^n} \frac{1_I}{|I|} \otimes [M_{\calD^m} \langle f, h_I \rangle_1 ]^2  \Big)^{1/2}\nu \Big\|_{L^p(\lambda)} \\
&= \Big\| \Big( \sum_{I \in \calD^n} \frac{1_I}{|I|} \otimes [M_{\calD^m} \langle f, h_I \rangle_1 ]^2  \Big)^{1/2} \Big\|_{L^p(\mu)} \lesssim_{[\mu]_{A_p}} \|f\|_{L^p(\mu)}.
\end{align*}
\end{proof}

We now take care of the remaining $U^b$ term.
This key term also arises in \cite{HPW} where it is actually omitted
by saying that it goes similarly as a certain other term (which does not arise at all in our decomposition). 
To handle this term we find it necessary to use somewhat sophisticated tools via bilinear sparse domination.

Similarly as in the shift case it is enough to fix $l \in \{1, \dots, k_1\}$ and estimate the term
\begin{equation}\label{eq:PPAv-Av}
\begin{split}
\sum_{K \in \calD^n}&
\sum_{\substack{L \in \calD^n \\ L^{(k_1-l)}=K}} \int_{\R^n} \frac{1_L(x_1)}{|L|} \sum_{\substack{I_1, I_2 \in \calD^n \\ I_1^{(l)}=L \\ I_2^{(k_2)} = K}} \\
& 
 \sum_{V \in \calD^m} \int_{\R^m} \frac{1_V(x_2)}{|V|}
\Big| a_{K, V, (I_i)} \Big\langle f_1,  h_{I_1}\otimes  \frac{1_V}{|V|}\Big\rangle
\langle f_2,  h_{I_2} \otimes  h_{V}\rangle\Big| \nu(x_1,x_2)\ud x_2 \ud x_1.
\end{split}
\end{equation}
From the sparse domination of bilinear paraproducts (see e.g. \cite{LMOV}) we can deduce (see Lemma 6.7 in \cite{LMV}) that
\begin{equation}\label{eq:ApplySparse}
\begin{split}
\int_{\R^m} &\Big( \sum_{V \in \calD^m} \big| a_{K, V, (I_i)} \bla \langle f_1, h_{I_1} \rangle_1 \bra_V \langle \langle f_2, h_{I_2}\rangle_1, h_V\rangle\big| \frac{1_V}{|V|} \Big)\rho \\
&\lesssim_{[\rho]_{A_{\infty}(\R^m)}} \frac{|I_1|^{1/2} |I_2|^{1/2}}{|K|} \int_{\R^m} M_{\calD^m} ( \langle f_1, h_{I_1} \rangle_1 ) M_{\calD^m}( \langle f_2, h_{I_2}\rangle_1)\rho.
\end{split}
\end{equation}
This requires knowing that we have
$
\sup_{V_0 \in \calD^m} \Big( \frac{1}{|V_0|} \sum_{\substack{V \in \calD^m \\ V \subset V_0}} |a_{K, V, (I_i)}|^2\Big)^{1/2} \le \frac{|I_1|^{1/2} |I_2|^{1/2}}{|K|}.
$
Recall the function $\varphi_{\calD^n, \calD^m}^1$ from Lemma \ref{lem:standardEst} and then notice the identity $M_{\calD^m} ( \langle f_1, h_{I_1} \rangle_1) = \langle \varphi_{\calD^n, \calD^m}^1 f_1 , h_{I_1} \rangle_1$. We 
now see that \eqref{eq:PPAv-Av} can be dominated in the $\lesssim_{[\mu]_{A_p}, [\lambda]_{A_p}}$ sense by
\begin{align*}
 \sum_{K \in \calD^n} \sum_{\substack{L \in \calD^n \\ L^{(k_1-l)}=K}} &\iint \frac{1_L}{|L|} \sum_{\substack{I_1, I_2 \in \calD^n \\ I_1^{(l)}=L \\ I_2^{(k_2)} = K}}
 \frac{|I_1|^{1/2} |I_2|^{1/2}}{|K|}   \langle \varphi_{\calD^n, \calD^m}^1 f_1 , h_{I_1} \rangle_1  \langle \varphi_{\calD^n, \calD^m}^1 f_2, h_{I_2}\rangle_1\nu \\
 &\le \sum_{K \in \calD^n} \sum_{\substack{L \in \calD^n \\ L^{(k_1-l)}=K}} \iint 1_L \langle |\Delta_{K, k_1}^1 \varphi_{\calD^n, \calD^m}^1 f_1| \rangle_{L,1}
 \langle |\Delta_{K, k_2}^1 \varphi_{\calD^n, \calD^m}^1 f_2| \rangle_{K,1} \nu \\
 &\le \sum_{K \in \calD^n} \iint M_{\calD^n}^1  \Delta_{K, k_1}^1 \varphi_{\calD^n, \calD^m}^1 f_1\cdot M_{\calD^n}^1 \Delta_{K, k_2}^1 \varphi_{\calD^n, \calD^m}^1 f_2 \cdot \nu \\
& \le \Big\| \Big( \sum_{K \in \calD^n} [M_{\calD^n}^1  \Delta_{K, k_1}^1 \varphi_{\calD^n, \calD^m}^1 f_1]^2 \Big)^{1/2} \Big\|_{L^p(\mu)} \\
&\hspace{3cm} \times  \Big\| \Big( \sum_{K \in \calD^n} [M_{\calD^n}^1  \Delta_{K, k_2}^1 \varphi_{\calD^n, \calD^m}^1 f_2]^2 \Big)^{1/2} \Big\|_{L^{p'}(\lambda^{1-p'})} \\
 & \lesssim_{[\mu]_{A_p}, [\lambda]_{A_p}} \|f_1\|_{L^p(\mu)} \|f_2\|_{L^{p'}(\lambda^{1-p'})}.
\end{align*}

\subsection{The full paraproduct case}
Depending on the form of $U$ (we have two genuinely different symmetries here), we get different terms in the expansion (following Section \ref{sec:marprod}) of
$\langle [b, U]f_1, f_2\rangle$. However, after minor thought (recall also that $U^b = 0$) the reader will understand
that the only type of term that we have not seen before is
\begin{equation}\label{eq:fullPOnly}
\sum_{\substack{K\in \calD^n \\  V\in \calD^m}} a_{K,V} \bla (b-\langle b \rangle_{K \times V})f_1\bra_{K \times V} \langle f_2, h_K \otimes h_V\rangle.
\end{equation}
To handle this via \eqref{eq:1WeightForModel} we introduce the following maximal function:
\[
M_{\calD^n, \calD^m}^b(f)=\sup_R \frac {1_R}{|R|}\int_R |b-\langle b \rangle_R||f|,
\]
where $R = I \times J \in \calD^n \times \calD^m$.
\begin{prop}\label{prop:bloomforMb}
Let $p \in (1,\infty)$ and $b\in \bmo(\nu)$, where $\mu, \lambda \in A_p$ and $\nu=\mu^{1/p} \lambda^{-1/p}$. Then we have
\[
\| M_{\calD^n, \calD^m}^b\|_{L^p(\mu)\rightarrow L^p(\lambda)}\lesssim_{[\mu]_{A_p}, [\lambda]_{A_p}} \|b\|_{\bmo(\nu)}.
\]
\end{prop}
\begin{proof}
There exists some $1<q<p$ such that $\mu, \lambda\in A_q$. Then H\"older's inequality implies that
$
\mu_0= \mu^{\frac qp} \lambda^{1-\frac qp}\in A_q.
$
Since $\mu_0^{\frac 1q}\lambda^{-\frac 1q}=\nu$, by the two-weight John-Nirenberg for little BMO \eqref{eq:2JN} we have for all $x$ and $R \ni x$ that
\begin{align*}
 \frac 1{|R|}\int_R |b-\langle b \rangle_R| |f|&\le \Big(  \frac 1{|R|} \int_R |b-\langle b\rangle_R |^{q'}\mu_0^{1-q'}\Big)^{\frac 1{q'}}\Big(  \frac 1{|R|} \int_R |f |^{q}\mu_0 \Big)^{\frac 1{q}}\\
 &\lesssim_{[\mu]_{A_p}, [\lambda]_{A_p}} \|b\|_{\bmo(\nu)} \frac{\lambda^{1-q'}(R)^{\frac 1{q'}}}{|R|}\Big(   \int_R |f |^{q}\mu_0 \Big)^{\frac 1{q}}\\
 &\lesssim_{[\lambda]_{A_p}} \|b\|_{\bmo(\nu)}\Big( \frac 1{\lambda(R)}  \int_R |f |^{q}\mu_0 \Big)^{\frac 1{q}}\\
 &\le \|b\|_{\bmo(\nu)} M_{\calD^n, \calD^m, \lambda}(|f|^q \mu_0 \lambda^{-1})(x)^{\frac 1q}.
\end{align*}
The claim now follows from the boundedness property $M_{\calD^n, \calD^m, \lambda} \colon L^{p/q}(\lambda) \to L^{p/q}(\lambda)$ 
and the observation that $(\mu_0\lambda^{-1})^{\frac pq} \lambda=\mu$. The first mentioned fact is non-trivial as $\lambda$ is not of product form -- but
it has been proved by R. Fefferman in \cite{Fe3} using the $A_{\infty}$ property of $\lambda$. For clarity we give a proof in our dyadic setting in Appendix \ref{app2}.
\end{proof}
Notice that
$$
\big|\bla (b-\langle b \rangle_{K \times V})f_1 \bra_{K \times V}\big| \lesssim \langle  M_{\calD^n, \calD^m}^b f_1 \rangle_{K \times V}
$$
so that \eqref{eq:1WeightForModel} gives that the absolute value of \eqref{eq:fullPOnly} can be dominated with
$$
C([\lambda]_{A_p}) \| M_{\calD^n, \calD^m}^b f_1 \|_{L^p(\lambda)} \|f_2\|_{L^{p'}(\lambda^{1-p'})} \lesssim_{[\mu]_{A_p}, [\lambda]_{A_p}} \|b\|_{\bmo(\nu)} \|f_1\|_{L^p(\mu)}  \|f_2\|_{L^{p'}(\lambda^{1-p'})}.
$$
Here the last estimate used Lemma \ref{prop:bloomforMb}.
We are done with the full paraproducts.

\section{Iterated commutators}
To study the Bloom type inequality for iterated commutators, we also need to consider the commutators of general paraproduct operators that
appear in Section \ref{sec:marprod}.
\begin{lem}\label{lem:commupib}
Let $\pi_b$ be $A_i(b,\cdot)$, $i=1,\cdots, 8$ or $a_j^1(b,\cdot)$, $a_j^2(b,\cdot)$, $j=1,2$. Suppose $b_1 \in \bmo(\nu^{\theta_1})$ and $b_2 \in \bmo(\nu^{\theta_2})$, where $\nu=\mu^{1/p}\lambda^{-1/p}$, $0\le \theta_1,\theta_2\le 1$, $\theta_1+\theta_2=1$ and $\mu, \lambda\in A_p$. Then
\begin{align*}
\|[b_2,\pi_{b_1}]\|_{L^p(\mu)\rightarrow L^p(\lambda)}\lesssim_{[\mu]_{A_p}, [\lambda]_{A_p}}  \|b_1\|_{\bmo(\nu^{\theta_1})}  \|b_2\|_{\bmo(\nu^{\theta_2})}.
\end{align*}
\end{lem}
\begin{proof}
With our existing tools there is no essential difference in the proof for different operators, and we e.g. choose $\pi_{b_1} = A_5(b_1, \cdot)$. We have
$$
\langle A_5(b_1, f_1), f_2\rangle = \sum_{I,J} \Big\langle b_1, \frac{1_I}{|I|} \otimes h_J \Big\rangle \langle f_1, h_I \otimes h_J\rangle \langle f_2, h_I \otimes h_Jh_J\rangle.
$$
Using our decomposition philosophy (treating $h_Jh_J$ as non-cancellative) we get
\begin{equation}\label{eq:eq1}
\begin{split}
\langle [b_2, &A_5(b_1, f_1)], f_2 \rangle = \sum_{i=1}^2 \langle A_5(b_1, f_1), a^1_i(b_2, f_2)\rangle
-  \sum_{i=1}^8 \langle A_5(b_1, A_i(b_2, f_1)),f_2\rangle \\
&+ \sum_{I,J} \Big\langle b_1, \frac{1_I}{|I|} \otimes h_J \Big\rangle \langle f_1, h_I \otimes h_J\rangle
\bla (\langle b_2 \rangle_{I,1} - \langle b_2 \rangle_{I \times J}) \langle f_2, h_{I}\rangle_1, h_Jh_J\bra.
\end{split}
\end{equation}
The first and second term are similar -- we only deal with the second one. We begin with the only reasonable step:
$$
|\langle A_5(b_1, A_i(b_2, f_1)),f_2\rangle| \le \|A_5(b_1, A_i(b_2, f_1))\|_{L^p(\lambda)} \|f_2\|_{L^{p'}(\lambda^{1-p'})}.
$$
We want to use the Bloom inequality for $A_5(b_1, \cdot)$ with $b_1 \in \bmo(\nu^{\theta_1})$. Thus, we write
$$
\nu^{\theta_1} = (\mu^{\theta_1}\lambda^{1-\theta_1})^{1/p} \lambda^{-1/p}, \, \textup{ where } \mu^{\theta_1}\lambda^{1-\theta_1}, \lambda \in A_p,
$$
and get
$$
\|A_5(b_1, A_i(b_2, f_1))\|_{L^p(\lambda)} \lesssim_{[\mu]_{A_p}, [\lambda]_{A_p}} \|b_1\|_{\bmo(\nu^{\theta_1})} \| A_i(b_2, f_1) \|_{L^p(\mu^{\theta_1}\lambda^{1-\theta_1})}.
$$
Then we write
$$
\nu^{\theta_2} = \mu^{1/p} (\mu^{\theta_1}\lambda^{1-\theta_1})^{-1/p}, \, \textup{ where } \mu, \mu^{\theta_1}\lambda^{1-\theta_1} \in A_p,
$$
and similarly get
$$
\| A_i(b_2, f_1) \|_{L^p(\mu^{\theta_1}\lambda^{1-\theta_1})} \lesssim_{[\mu]_{A_p}, [\lambda]_{A_p}} \|b_2\|_{\bmo(\nu^{\theta_2})} \|f_1\|_{L^p(\mu)}.
$$

For the third term in \eqref{eq:eq1} we begin with Lemma \ref{lem:PhiNuLemma}, which gives us that
$$
|\bla (\langle b_2 \rangle_{I,1} - \langle b_2 \rangle_{I \times J}) \langle f_2, h_{I}\rangle_1, h_Jh_J\bra|
\lesssim_{[\mu]_{A_p}, [\lambda]_{A_p}} \| b_2 \|_{\bmo(\nu^{\theta_2})}\Big\langle \varphi_{\calD^n, \calD^m}^{\nu^{\theta_2},1 } f_2, h_I \otimes \frac{1_J}{|J|} \Big\rangle.
$$
Then writing
$$
\nu^{\theta_1} = \mu^{1/p} (\mu^{1-\theta_1} \lambda^{\theta_1})^{-1/p}, \, \textup{ where } \mu, \mu^{1-\theta_1} \lambda^{\theta_1} \in A_p,
$$
we get (again using the known Bloom for $A_5(b_1, \cdot)$) that the absolute value of the third term in \eqref{eq:eq1}
can be dominated in the $\lesssim_{[\mu]_{A_p}, [\lambda]_{A_p}}$ sense by
$$
\|b_1\|_{\bmo(\nu^{\theta_1})} \|f_1\|_{L^p(\mu)} \|\varphi_{\calD^n, \calD^m}^{\nu^{\theta_2},1 } f_2\|_{L^{p'}((\mu^{1-\theta_1} \lambda^{\theta_1})^{1-p'})}.
$$
Then using Lemma \ref{lem:PhiNuLemma} together with the identity
$$
\nu^{\theta_2} = (\lambda^{1-p'})^{1/p'} ( (\mu^{1-\theta_1} \lambda^{\theta_1})^{1-p'} )^{-1/p'},  \, \textup{ where } \lambda^{1-p'}, (\mu^{1-\theta_1} \lambda^{\theta_1})^{1-p'} \in A_{p'},
$$
we get
$$
\|\varphi_{\calD^n, \calD^m}^{\nu^{\theta_2},1 } f_2\|_{L^{p'}((\mu^{1-\theta_1} \lambda^{\theta_1})^{1-p'})} \lesssim \|f_2\|_{L^{p'}(\lambda^{1-p'})}.
$$
We are done.
\end{proof}
\subsection{The shift case}
We show that if $U = U^{k,v}$ is a shift then
\begin{align*}
|\langle [b_2,& [b_1, U]]f_1, f_2\rangle| \\
&\lesssim_{[\mu]_{A_p}, [\lambda]_{A_p}} \|b_1\|_{\bmo(\nu^{\theta_1})}  \|b_2\|_{\bmo(\nu^{\theta_2})} (1+\max(k_i, v_i))^2 \|f_1\|_{L^p(\mu)} \|f_2\|_{L^{p'}(\lambda^{1-p'})}.
\end{align*}
We recall Equation \eqref{eq:Com1ofShift} with $b = b_1$. In the iterated commutator $\langle [b_2, [b_1, U]]f_1, f_2\rangle$ the first term of \eqref{eq:Com1ofShift} leads to the need to study
$$
\langle Uf_1, A_i(b_1, b_2f_2) \rangle - \langle U(b_2f_1), A_i(b_1, f_2)\rangle,
$$
which can be written (by adding and subtracting the obvious term) in the form
$$
-\langle Uf_1, [b_2, A_i(b_1, \cdot)]f_2 \rangle + \langle [b_2,U]f_1, A_i(b_1, f_2)\rangle.
$$
We have using \eqref{eq:1WeightForModel} and Lemma \ref{lem:commupib} that
\begin{align*}
|\langle Uf_1, [b_2, A_i(b_1, \cdot)]f_2\rangle| &\le \|Uf_1\|_{L^p(\mu)} \| [b_2, A_i(b_1, \cdot)]f_2 \|_{L^{p'}(\mu^{1-p'})} \\
& \lesssim_{[\mu]_{A_p}, [\lambda]_{A_p}}  \|b_1\|_{\bmo(\nu^{\theta_1})}  \|b_2\|_{\bmo(\nu^{\theta_2})} \|f_1\|_{L^p(\mu)} \|f_2\|_{L^{p'}(\lambda^{1-p'})}.
\end{align*}
On the other hand, using the known Bloom type inequality for the first order commutator $[b_2,U]$ and also for $A_i(b_1, \cdot)$, we get arguing
analogously as in the proof of Lemma \ref{lem:commupib} that
\begin{align*}
|\langle [b_2,U]f_1,& A_i(b_1, f_2)\rangle| \\
&\lesssim_{[\mu]_{A_p}, [\lambda]_{A_p}}  \|b_1\|_{\bmo(\nu^{\theta_1})}  \|b_2\|_{\bmo(\nu^{\theta_2})} (1+\max(k_i, v_i)) \|f_1\|_{L^p(\mu)} \|f_2\|_{L^{p'}(\lambda^{1-p'})}.
\end{align*}
We have thus handled the contribution of the first term of \eqref{eq:Com1ofShift} to $\langle [b_2, [b_1, U]]f_1, f_2\rangle$.
The contribution of the second term of \eqref{eq:Com1ofShift} to $\langle [b_2, [b_1, U]]f_1, f_2\rangle$ is handled in the same way.

Therefore, we are only left with bounding the contribution of the third term of \eqref{eq:Com1ofShift} to $\langle [b_2, [b_1, U]]f_1, f_2\rangle$, i.e. bounding
$\langle [b_2, U^{b_1}]f_1, f_2\rangle$.
Expanding this as in \eqref{eq:Com1ofShift}, we are left with some
terms that can be handled using the already known Bloom type inequality for $U^{b_1}$ and the Bloom type
inequality for $A_i(b_2, \cdot)$, and also with the new term
\begin{align*}
\langle U^{b_1, b_2}f_1, f_2 \rangle :=
\sum_{\substack{K \in \calD^n \\ V \in \calD^m}}
\sum_{\substack{I_1, I_2 \in \calD^n \\ I_i^{(k_i)} = K}} &
\sum_{\substack{J_1, J_2 \in \calD^m \\ J_j^{(v_j)} = V}} a_{K, V, (I_i), (J_i)} [\langle b_1 \rangle_{I_2 \times J_2} - \langle b_1 \rangle_{I_1 \times J_1}]\\
&[\langle b_2 \rangle_{I_2 \times J_2} - \langle b_2 \rangle_{I_1 \times J_1}]
\langle f_1,  h_{I_1} \otimes  h_{J_1}\rangle
\langle f_2,  h_{I_2} \otimes  h_{J_2}\rangle.
\end{align*}

To finish the shift case, we now bound $U^{b_1, b_2}$. We use \eqref{eq:av-av-split} with $b = b_1$ and $b = b_2$.
When we multiply these together, we get multiple different terms -- we pick two representative ones
\begin{equation}\label{eq:s1}
[\langle b_1 \rangle_{K \times V} -  \langle b_1 \rangle_{K \times J_1}][\langle b_2 \rangle_{K \times J_1} - \langle b_2 \rangle_{I_1 \times J_1}]
\end{equation}
and
\begin{equation}\label{eq:s2}
[\langle b_1 \rangle_{I_2 \times J_2} - \langle b_1 \rangle_{K \times J_2}][\langle b_2 \rangle_{K \times J_1} - \langle b_2 \rangle_{I_1 \times J_1}].
\end{equation}
The point is that in the first case we only have $I_1,J_1$ appearing in both terms (and no $I_2, J_2$), and the other one is a mixed case.
Nevertheless, they can in fact be handled with completely analogous estimates.
Therefore, we only deal with \eqref{eq:s1}.

We use analogous estimates to \eqref{eq:av-av}, which leads to the need to bound
\begin{equation}\label{eq:ShiftAv-Av-iterated}
\begin{split}
\sum_{\substack{K \in \calD^n \\ V \in \calD^m}}
&\sum_{\substack{H \in \calD^m \\ H^{(v_1-h)} = V}}
\frac{\nu^{\theta_1}(K \times H)}{|K \times H|}
\sum_{\substack{L \in \calD^n \\ L^{(k_1-l)}=K}}
 \sum_{\substack{J_1 \in \calD^m \\ J_1^{(h)} = H }}
  \\
& \frac{\nu^{\theta_2}(L \times J_1)}{|L \times J_1|}
\sum_{\substack{I_1, I_2 \in \calD^n \\ I_1^{(l)}=L \\ I_2^{(k_2)} = K}} \sum_{\substack{ J_2 \in \calD^m \\ J_2^{(v_2)} = V}}
\big| a_{K, V, (I_i), (J_i)} \langle f_1,  h_{I_1} \otimes  h_{J_1}\rangle
\langle f_2,  h_{I_2} \otimes  h_{J_2}\rangle\big|,
\end{split}
\end{equation}
where $l \in \{1, \ldots, k_1\}$ and $h \in \{1, \ldots, v_1\}$.
The second line of \eqref{eq:ShiftAv-Av-iterated} is dominated by
$$
 \iint 1_{L \times J_1} M_{\calD^n, \calD^m}(\Delta_{K \times V}^{k_1, v_1} f_1) \langle |\Delta_{K \times V}^{k_2, v_2} f_2| \rangle_{K \times V} \nu^{\theta_2}.
$$
Therefore, continuing in the same way \eqref{eq:ShiftAv-Av-iterated} can be dominated with
$$
\sum_{\substack{K \in \calD^n \\ V \in \calD^m}} \iint M_{\calD^n, \calD^m}(M_{\calD^n, \calD^m}(\Delta_{K \times V}^{k_1, v_1} f_1)\nu^{\theta_2})
M_{\calD^n, \calD^m}(\Delta_{K \times V}^{k_2, v_2} f_2) \nu^{\theta_1}.
$$
Writing $\nu^{\theta_1} = (\mu^{\theta_1}\lambda^{\theta_2})^{1/p} \lambda^{-1/p}$ we can dominate this with
\begin{align*}
\Big\| \Big( \sum_{\substack{K \in \calD^n \\ V \in \calD^m}} [M_{\calD^n, \calD^m}(&M_{\calD^n, \calD^m}(\Delta_{K \times V}^{k_1, v_1} f_1)\nu^{\theta_2})]^2\Big)^{1/2} \Big\|_{L^p(\mu^{\theta_1}\lambda^{\theta_2})}\\
&\times
 \Big\| \Big( \sum_{\substack{K \in \calD^n \\ V \in \calD^m}} [M_{\calD^n, \calD^m}\Delta_{K \times V}^{k_2, v_2} f_2]^2 \Big)^{1/2}\Big\|_{L^{p'}(\lambda^{1-p'})}.
\end{align*}
We can conclude the case \eqref{eq:s1}
by using Fefferman--Stein, square function estimates and also noting that $\nu^{\theta_2 p } \mu^{\theta_1}\lambda^{\theta_2} = \mu$.
We are done with the shift case.
\subsection{The partial paraproduct case}
We show that if $U = U^{k}$ is a partial paraproduct of the form \eqref{eq:FormOfPP}, then we have
\begin{align*}
|\langle [b_2,& [b_1, U]]f_1, f_2\rangle| \\
&\lesssim_{[\mu]_{A_p}, [\lambda]_{A_p}} \|b_1\|_{\bmo(\nu^{\theta_1})}  \|b_2\|_{\bmo(\nu^{\theta_2})} (1+\max(k_1, k_2))^2 \|f_1\|_{L^p(\mu)} \|f_2\|_{L^{p'}(\lambda^{1-p'})}.
\end{align*}
Recall \eqref{eq:partial} with $b=b_1$. The contributions of the first two terms of \eqref{eq:partial} to $\langle [b_2, [b_1, U]]f_1, f_2\rangle$ are handled using the same
general argument that we used with shifts.

We now bound the contribution of the third term of \eqref{eq:partial} to $\langle [b_2, [b_1, U]]f_1, f_2\rangle$, i.e. we bound
$\langle [b_2, U^{b_1}]f_1, f_2\rangle$. Expanding this as in \eqref{eq:partial}, the first two terms can be handled using
the already known Bloom type inequality for $U^{b_1}$ and the Bloom type inequality for $A_i(b_2, \cdot)$ and $a_i^1(b_2, \cdot)$, while
the last term can be handled using Lemma \ref{lem:PhiNuLemma} and the Bloom for $U^{b_1}$. Therefore, we are again facing the need to handle
$U^{b_1, b_2}$. 

Recall that $U=U^k$ is of the form \eqref{eq:FormOfPP}. 
When written out, $U^{b_1, b_2}$ includes terms of the form 
$(\langle b_1 \rangle_{I_2\times V}-\langle b_1 \rangle_{I_1\times V})(\langle b_2 \rangle_{I_2\times V}-\langle b_2 \rangle_{I_1\times V})$.
As before, we split 
$$
\langle b_1 \rangle_{I_2\times V}-\langle b_1 \rangle_{I_1 \times V}
=[\langle b_1 \rangle_{I_2\times V}-\langle b_1 \rangle_{K\times V}]+[\langle b_1 \rangle_{K\times V}-\langle b_1 \rangle_{I_1\times V}],
$$ 
and similarly with the function $b_2$. These multiplied together divides $U^{b_1, b_2}$ into four parts, which are handled in the same way. 
To complement the case we handled with shifts (where we chose \eqref{eq:s1} instead of \eqref{eq:s2}), we choose here the part coming from the terms 
$$
(\langle b_1 \rangle_{K\times V}-\langle b_1 \rangle_{I_1\times V}) (\langle b_2 \rangle_{I_2\times V}-\langle b_2 \rangle_{K\times V}).
$$
We apply the estimate \eqref{eq:av-av}, and see that it suffices to bound the term
\begin{equation}\label{eq:IterPPParticular}
\begin{split}
\sum_{K \in \calD^n } \sum_{\substack{L_1,L_2 \in \calD^n \\ L_i^{(k_i-l_i)}=K}}
\sum_{\substack{I_1,I_2 \in \calD^n \\ I_i^{(l_i)}=L_i}} \sum_{V \in \calD^m}
&\frac{\nu^{\theta_1}(L_1 \times V)}{|L_1 \times V|}\frac{\nu^{\theta_2}(L_2 \times V)}{|L_2 \times V|} \\
&\times \Big|a_{K,V,(I_i)} \Big\langle f_1,  h_{I_1}\otimes  \frac{1_V}{|V|}\Big\rangle
\langle f_2,  h_{I_2} \otimes  h_{V}\rangle \Big|,
\end{split}
\end{equation}
where $l_i \in \{1, \dots, k_i\}$, $i \in \{1, 2\}$.

If $\rho_1, \rho_2 \in A_2(\R^m)$, then
\begin{equation*}
\langle \rho_1\rangle_J \langle \rho_2\rangle_J\le [\rho_1]_{A_2}[\rho_2]_{A_2}\langle \rho_1\rho_2 \rangle_J.
\end{equation*}
Indeed, by H\"older's inequality there holds for any cube $J \subset \R^m$ that
\[
1=\frac 1{|J|}\int_J  (\rho_1 \rho_2)^{\frac 13} \rho_1^{-\frac 13} \rho_2^{-\frac 13}
\le  \langle \rho_1\rho_2\rangle_J^{\frac 13}\langle \rho_1^{-1}\rangle_J^{\frac 13} \langle \rho_2^{-1}\rangle_J^{\frac 13}, 
\]
which combined with $\langle \rho_i ^{-1} \rangle_J \le [\rho_i]_{A_2} \langle \rho_i  \rangle_J^{-1}$ gives the claim. 
If $L_1,L_2 \in \calD^n$ and $V \in \calD^m$, this shows that
\begin{equation}\label{eq:IntoOneWeight}
\begin{split}
\frac{\nu^{\theta_1}(L_1\times V)}{|L_1 \times V|} \frac{\nu^{\theta_2}(L_2\times V)}{|L_2 \times V|} 
&\le [\langle \nu^{\theta_1}\rangle_{L_1,1}]_{A_2}[\langle \nu^{\theta_2}\rangle_{L_2,1}]_{A_2}
\big\langle  \langle \nu^{\theta_1}\rangle_{L_1,1} \langle \nu^{\theta_2}\rangle_{L_2,1}\big\rangle_V\\
&\le [\nu]_{A_2}\big\langle  \langle \nu^{\theta_1}\rangle_{L_1,1} \langle \nu^{\theta_2}\rangle_{L_2,1}\big\rangle_V.
\end{split}
\end{equation}

We turn to \eqref{eq:IterPPParticular}. If $K,L_1,L_2,I_1$ and $I_2$ are as in \eqref{eq:IterPPParticular}, then
applying \eqref{eq:IntoOneWeight} one sees that the inner sum over $V \in \calD^m$ is less than $[\nu]_{A_2}$ multiplied by
\begin{equation*}
\begin{split}
&\int_{\R^m}  \sum_{V \in \calD^m}  
 \big| a_{K, V, (I_i)} \bla \langle f_1, h_{I_1} \rangle_1 \bra_V \langle \langle f_2, h_{I_2}\rangle_1, h_V\rangle\big| 
\frac{1_V}{|V|} \langle \nu^{\theta_1}\rangle_{L_1,1} \langle \nu^{\theta_2}\rangle_{L_2,1} \\
&   \lesssim _{[\langle \nu^{\theta_1}\rangle_{L_1,1} \langle \nu^{\theta_2}\rangle_{L_2,1}]_{A_\infty}}
 \frac{|I_1|^{1/2} |I_2|^{1/2}}{|K|} \int_{\R^m} M_{\calD^m} ( \langle f_1, h_{I_1} \rangle_1 ) M_{\calD^m}( \langle f_2, h_{I_2}\rangle_1)
 \langle \nu^{\theta_1}\rangle_{L_1,1} \langle \nu^{\theta_2}\rangle_{L_2,1},
\end{split}
\end{equation*}
where we used the application of sparse domination as in \eqref{eq:ApplySparse}.
Notice that Theorem 2.1 in \cite{CN} implies that $\langle \nu \rangle_{L_1,1}^{\theta_1} \lesssim_{[\nu]_{A_2}} \langle \nu^{\theta_1} \rangle_{L_1,1}$ (while the other
direction is trivial by H\"older's inequality). This implies that $[\langle \nu^{\theta_1}\rangle_{L_1,1} \langle \nu^{\theta_2}\rangle_{L_2,1}]_{A_2} \le C([\nu]_{A_2})$.

Recall the identity $M_{\calD^m} ( \langle f_i, h_{I_i} \rangle_1 )=\langle \varphi_{\calD^n, \calD^m}^1 f_i , h_{I_i} \rangle_1$.
We sum the last estimate over $K,L_1,L_2,I_1$ and $I_2$, and move the summations inside the integral over $\R^m$. 
Then, for a fixed $x_2 \in \R^m$,
we can use one parameter estimates in the same spirit that we used in connection with \eqref{eq:ShiftAv-Av-iterated}. 
This shows that \eqref{eq:IterPPParticular} is dominated by $C([\mu]_{A_p}, [\lambda]_{A_p})$ multiplied by
\begin{equation*}
 \iint_{\R^{n+m}} \sum_{K \in \calD^n}
M^1_{\calD^n} (  \Delta^1_{K,k_1} \varphi^1_{\calD^n, \calD^m} f_1) 
M^1_{\calD^n}(M^1_{\calD^n} (  \Delta^1_{K,k_2} \varphi^1_{\calD^n, \calD^m} f_2) \nu^{\theta_2}) \nu^{ \theta_1}.
\end{equation*}
From here the estimate can be concluded by familiar steps. This ends our study of $U^{b_1, b_2}$.

To finish our treatment of partial paraproducts, we need to bound the contribution of the last term of \eqref{eq:partial} to $\langle [b_2, [b_1, U]]f_1, f_2\rangle$. Expanding using our usual rules leads us to the following sum of terms
\begin{align*}
&\sum_{i=1}^8\sum_{\substack{K \in \calD^n \\ V \in \calD^m}} \sum_{\substack{I_1, I_2 \in \calD^n \\ I_1^{(k_1)}= I_2^{(k_2)} = K}} a_{K, V, (I_i)}
\bla (\langle b_1 \rangle_{I_1, 1} - \bla b_1 \rangle_{I_1 \times V})\langle f_1, h_{I_1}\rangle_1 \bra_V\langle A_i(b_2,f_2),  h_{I_2} \otimes  h_{V}\rangle\\
&-\sum_{i=1}^2\sum_{\substack{K \in \calD^n \\ V \in \calD^m}} \sum_{\substack{I_1, I_2 \in \calD^n \\ I_1^{(k_1)}= I_2^{(k_2)} = K}} a_{K, V, (I_i)}
\bla (\langle b_1 \rangle_{I_1, 1} - \bla b_1 \rangle_{I_1 \times V})\langle a_i^1(b_2,f_1), h_{I_1}\rangle_1 \bra_V\langle f_2,  h_{I_2} \otimes  h_{V}\rangle\\
&+\sum_{\substack{K \in \calD^n \\ V \in \calD^m}} \sum_{\substack{I_1, I_2 \in \calD^n \\ I_1^{(k_1)}= I_2^{(k_2)} = K}} a_{K, V, (I_i)}
\bla (\langle b_1 \rangle_{I_1, 1} - \bla b_1 \rangle_{I_1 \times V})( \bla b_2 \rangle_{I_1 \times V}-\langle b_2 \rangle_{I_1, 1} )\langle f_1, h_{I_1}\rangle_1 \bra_V\\
&\hspace{12cm}\times\langle f_2,  h_{I_2} \otimes  h_{V}\rangle\\
&+\sum_{\substack{K \in \calD^n \\ V \in \calD^m}} \sum_{\substack{I_1, I_2 \in \calD^n \\ I_1^{(k_1)}= I_2^{(k_2)} = K}} a_{K, V, (I_i)} [\bla b_2 \rangle_{I_2 \times V}-\bla b_2 \rangle_{I_1 \times V}]
\bla (\langle b_1 \rangle_{I_1, 1} - \bla b_1 \rangle_{I_1 \times V})\langle f_1, h_{I_1}\rangle_1 \bra_V\\
&\hspace{12cm}\times\langle f_2,  h_{I_2} \otimes  h_{V}\rangle\\
&= I + II + III + IV.
\end{align*}
The estimates for $I$ and $II$ follow quite directly from \eqref{eq:1WeightForModel} and lemmas \ref{lem:basicAa} and \ref{lem:PhiNuLemma}.
The term $IV$ can be handled using Lemma \ref{lem:PhiNuLemma} and the Bloom type inequality of $U^{b_2}$. The term $III$ requires a more careful treatment, which we will now proceed to give.

We will use the following one parameter estimate 
\begin{align*}
\frac 1{|J|}\int_J |b_1-\langle b_1 \rangle_J| &|b_2-\langle b_2 \rangle_J|  |g|\lesssim [w_1]_{A_\infty} [w_2]_{A_\infty} \|b_1\|_{\BMO(w_1)}\|b_2\|_{\BMO(w_2)}\\
&\times\frac 1{|J|}\int_{ J} [M_{\calD^m}(M_{\calD^m}(g)w_2)w_1+M_{\calD^m}(M_{\calD^m}(g)w_1)w_2].
\end{align*}
We can prove this by using \eqref{eq:lor} again:
\begin{align*}
\int_J &|b_1-\langle b_1 \rangle_J| |b_2-\langle b_2 \rangle_J|  |g|\\
&\lesssim  \sum_{\substack{P\in \mathcal S_1\\ P\subset J}}\sum_{\substack{Q\in \mathcal S_2\\ Q\subset J}} \big\langle|b_1-\langle b_1\rangle_P|\big\rangle_P\big\langle|b_2-\langle b_2\rangle_Q|\big\rangle_Q \int_{P\cap Q} |g|\\
&\le \|b_1\|_{\BMO(w_1)}\|b_2\|_{\BMO(w_2)}\sum_{\substack{P\in \mathcal S_1\\ P\subset J}}\sum_{\substack{Q\in \mathcal S_2\\ Q\subset P}} \langle w_1\rangle_P\langle w_2\rangle_Q \int_{ Q} |g|\\
&+\|b_1\|_{\BMO(w_1)}\|b_2\|_{\BMO(w_2)}\sum_{\substack{Q\in \mathcal S_2\\ Q\subset J}} \sum_{\substack{P\in \mathcal S_1\\ P\subset Q}}\langle w_1\rangle_P\langle w_2\rangle_Q \int_{ P} |g|\\
&\lesssim [w_2]_{A_\infty} \|b_1\|_{\BMO(w_1)}\|b_2\|_{\BMO(w_2)}\sum_{\substack{P\in \mathcal S_1\\ P\subset J}}\langle w_1\rangle_P\int_{ P} M_{\calD^m}(g)w_2\\
&+[w_1]_{A_\infty}\|b_1\|_{\BMO(w_1)}\|b_2\|_{\BMO(w_2)}\sum_{\substack{Q\in \mathcal S_2\\ Q\subset J}} \langle w_2\rangle_Q \int_{ Q} M_{\calD^m}(g)w_1.
\end{align*}
The desired one parameter estimate follows from this as previously. Define now
\[
\varphi_{\calD^n, \calD^m}^{\nu_1, \nu_2,1} f_1=\sum_{I\in \calD^n} h_I \otimes M_{\calD^m}(M_{\calD^m}(\langle f_1, h_{I}\rangle_1)\langle\nu_1 \rangle_{I,1})\langle\nu_2\rangle_{I,1},
\]
and notice that we get
\begin{align*}
|\bla (\langle b_1 \rangle_{I_1, 1}& - \bla b_1 \rangle_{I_1 \times V})( \bla b_2 \rangle_{I_1 \times V}-\langle b_2 \rangle_{I_1, 1} )\langle f_1, h_{I_1}\rangle_1 \bra_V|\\
&\lesssim_{[\mu]_{A_p}, [\lambda]_{A_p}} \|b_1\|_{\bmo(\nu^{\theta_1})}  \|b_2\|_{\bmo(\nu^{\theta_2})} \Big\langle \varphi_{\calD^n, \calD^m}^{\nu^{\theta_1}, \nu^{\theta_2},1} f_1+\varphi_{\calD^n, \calD^m}^{\nu^{\theta_2}, \nu^{\theta_1},1} f_1, h_{I_1} \otimes \frac{1_V}{|V|} \Big\rangle.
\end{align*}
It is not hard to show (similarly as in Lemma \ref{lem:PhiNuLemma}) that
\begin{equation}\label{eq:iteratedvarphi}
\big\| \varphi_{\calD^n, \calD^m}^{\nu^{\theta_1}, \nu^{\theta_2},1} f_1\big\|_{L^p(\lambda)} \lesssim_{[\mu]_{A_p}, [\lambda]_{A_p}} \|f_1\|_{L^p(\mu)}. 
\end{equation}
This, together with \eqref{eq:1WeightForModel}, ends our treatment of the term $III$. We are done with the partial paraproducts.

\subsection{The full paraproduct case}
The only term that arises here, which cannot be handled using exactly the same arguments that we have seen above with shifts and partial paraproducts, is
\begin{align*}
\sum_{\substack{K\in \calD^n \\  V\in \calD^m}} 
  a_{K,V} \bla  (b_1 - \langle b_1\rangle_{K\times V})(b_2 - \langle b_2 \rangle_{K\times V}) f_1 \bra_{K \times V}  \langle f_2,  {h}_{K}\otimes {h}_{V} \rangle.
\end{align*}
The natural maximal function is now
\[
M_{\calD^n, \calD^m}^{b_1,b_2} f=\sup_{R \in \calD^n \times \calD^m} \frac{1_R}{|R|}\int_R |b_1-\langle b_1 \rangle_R| |b_2-\langle b_2 \rangle_R| |f|.
\]
We have
\[
\| M_{\calD^n, \calD^m}^{b_1, b_2}\|_{L^p(\mu)\rightarrow L^p(\lambda)}\lesssim_{[\mu]_{A_p}, [\lambda]_{A_p}} \|b_1\|_{\bmo(\nu^{\theta_1})}\|b_2\|_{\bmo(\nu^{\theta_2})}
\]
as in Proposition \ref{prop:bloomforMb}. Indeed, using the same notation, and noticing that
$\mu_0^{1-q'},\lambda^{1-q'}\in A_{q'}\subset A_{q'/\theta}$ for any $\theta\in [0,1]$, we then have, for $\theta_1\in (0,1)$ and $x\in R$, that
\begin{align*}
& \frac 1{|R|}\int_R |b_1-\langle b_1 \rangle_R|  |b_2-\langle b_2 \rangle_R|  |f|\\
&\le\Big( \frac 1{|R|}\int_R |f|^q \mu_0\Big)^{\frac 1q}  \Big(\frac 1{|R|}\int_R |b_1-\langle b_1 \rangle_R|^{\frac{q'}{\theta_1}} \mu_0^{1-q'}\Big)^{\frac {\theta_1}{q'}} \Big(\frac 1{|R|}\int_R |b_2 -\langle b_2  \rangle_R|^{\frac{q'}{\theta_2}}\mu_0^{1-q'}\Big)^{\frac {\theta_2}{q'}}\\
&\lesssim_{[\mu]_{A_p}, [\lambda]_{A_p}} \Big( \frac 1{|R|}\int_R |f|^q \mu_0\Big)^{\frac 1q}  \Big( \frac{ \lambda^{1-q'}(R)}{|R|} \Big)^{\frac {\theta_1}{q'}}\|b_1\|_{\bmo(\nu^{\theta_1})}\Big( \frac{ \lambda^{1-q'}(R)}{|R|} \Big)^{\frac {\theta_2}{q'}}\|b\|_{\bmo(\nu^{\theta_2})}\\
&\lesssim \|b_1\|_{\bmo(\nu^{\theta_1})} \|b\|_{\bmo(\nu^{\theta_2})}\Big( \frac 1{\lambda(R)}  \int_R |f |^{q}\mu_0 \Big)^{\frac 1{q}}.
\end{align*}
After this we can conclude as previously.
We still comment on the case $\theta_1=0$. Let $s=q'/{t'}$, where $t = t([\lambda]_{A_p}) \in (q, \frac{p+q}{2q'}+1)$ (these restrictions imply that $s > 1$ and $p/st > 1$)
will be chosen later to be close enough to $q$. Set $\mu_1^{1-t'}=\mu_0^{1-q'}$. For $x\in R$ we have
\begin{align*}
& \frac 1{|R|}\int_R |b_1-\langle b_1 \rangle_R|  |b_2-\langle b_2 \rangle_R|  |f|\\
&\le \Big(\frac 1{|R|}\int_R |b_1-\langle b_1\rangle_R|^{s'} \Big)^{\frac 1{s'}} \Big( \frac 1{|R|}\int_R |b_2-\langle b_2\rangle_R|^{st'}\mu_1^{1-t'}\Big)^{\frac 1{st'}}\Big( \frac 1{|R|}\int_R |f|^{st}\mu_1 \Big)^{\frac 1{st}}\\
&\lesssim_{[\mu]_{A_p}, [\lambda]_{A_p}}\|b_1\|_{\bmo} \|b_2\|_{\bmo(\nu)}\Big(\frac{\lambda^{1-q'}(R)}{|R|}\Big)^{\frac 1{q'}}\Big( \frac 1{|R|}\int_R |f|^{st}\mu_1 \Big)^{\frac 1{st}}\\
&\le \|b_1\|_{\bmo} \|b_2\|_{\bmo(\nu)} [M_{\calD^n, \calD^m,\lambda}(M_{\calD^n,\calD^m}(f^{st}\mu_1)^{\frac q{st}} \lambda^{-1})(x)]^{\frac 1q}.
\end{align*}
To finish, we prove that $\lambda^{-\frac pq+1}\in A_{\frac p{st}}$. 
We first prove $\lambda^{-\frac pq+1}\in A_{\frac p{q}}$. We can take $q$ close enough to $p$ such that $(p/q)'>p$.
Using H\"older's inequality
with the exponent $u$, where $1/u := (p-1)(p/q-1) < (p-1)(p'-1) = 1$, we get
\begin{align*}
\frac 1{|R|}\int_R \lambda^{-\frac pq+1} \Big(\frac 1{|R|}\int_R \lambda  \Big)^{\frac pq-1}\le \Big(\frac 1{|R|}\int_R \lambda^{-\frac1{p-1}} \Big)^{ (p-1)(\frac pq-1)} \Big(\frac 1{|R|}\int_R \lambda  \Big)^{\frac p{q}-1} \le [\lambda]_{A_p}^{\frac pq-1}.
\end{align*}
Finally, we choose $t$ very close to $q$ -- notice that if $t \to q$ then $s \to 1$ and so $p/{st}\rightarrow p/q$.
Using the open property of $A_{p/q}$ weights we conclude that $\lambda^{-\frac pq+1}\in A_{\frac p{st}}$.
Using this we can end the proof.

\begin{appendix}
\section{Embedding $\bmo(w) \subset \BMO_{\textup{prod}}(w)$}\label{app1}
We give a proof of the embedding -- we thank Prof. T. Hyt\"onen for giving us an outline of the proof.
\begin{prop}
We have $\|b\|_{\BMO_{\textup{prod}}^{\calD^n, \calD^m}(w)} \lesssim_{[w]_{A_2}} \|b\|_{\bmo_{\calD^n, \calD^m}(w)}$ if $w \in A_2(\R^n \times \R^m)$.
\end{prop}
\begin{proof}
First, using the $\ell^2$-valued Kahane--Khintchine inequality notice that
\begin{align*}
\Big( \mathop{\sum_{I \in \calD^n}}_{J \in \calD^m} |\Delta_{I \times J} f(x)|^2 \Big)^{1/2} 
&= \Big( \E \sum_{I} \Big |  \sum_J \epsilon_J \Delta_{I \times J} f(x) \Big|^2\Big)^{1/2} \\
&\sim \E  \Big( \sum_{I} \Big | \Delta_I \Big( \sum_J \epsilon_J \Delta_J^2 f(\cdot, x_2)\Big)(x_1) \Big|^2 \Big)^{1/2}.
\end{align*}
Taking $L^1(w)$ norm, and using the known (see \cite{CWW}) lower bound
$$
\|g\|_{L^p(\rho)} \lesssim_{[\rho]_{A_{\infty}}} \Big\|\Big( \sum_I |\Delta_I g|^2 \Big)^{1/2}\Big\|_{L^p(\rho)}, \qquad p \in (0, \infty),
$$
we get
\begin{align*}
\Big\| \Big( \mathop{\sum_{I \in \calD^n}}_{J \in \calD^m} |\Delta_{I \times J} f|^2 \Big)^{1/2} \Big\|_{L^1(w)} 
&\gtrsim_{[w]_{A_2}} \iint_{\R^{n+m}} \E' \Big| \sum_J \epsilon_J' \Delta_J^2 f(x_1, x_2)  \Big| w(x_1,x_2) \ud x_1 \ud x_2 \\
&\sim \Big\| \Big( \sum_{J \in \calD^m} |\Delta_J^2 f|^2 \Big)^{1/2} \Big\|_{L^1(w)}.
\end{align*}
Next, using the weighted one parameter $H^1-\BMO$ duality result (i.e. the one parameter analog of \eqref{eq:WeightedH1BMO}), see Wu \cite{Wu}, we get
\begin{equation}\label{eq:bfest}
\begin{split}
\Big| \iint_{\R^{n+m}} b(x) f(x) \ud x \Big| &\lesssim_{[w]_{A_2}} \|b\|_{\bmo_{\calD^n, \calD^m}(w)} \Big\| \Big( \sum_{J \in \calD^m} |\Delta_J^2 f|^2 \Big)^{1/2} \Big\|_{L^1(w)} \\
&\lesssim_{[w]_{A_2}} \|b\|_{\bmo_{\calD^n, \calD^m}(w)}\Big\| \Big( \mathop{\sum_{I \in \calD^n}}_{J \in \calD^m} |\Delta_{I \times J} f|^2 \Big)^{1/2} \Big\|_{L^1(w)}.
\end{split}
\end{equation}
Borrowing a calculation from \cite{HPW} we can conclude the proof. Indeed, for all $\Omega$ we have
\begin{align*}
\Big(& \mathop{\sum_{I \in \calD^n, J \in \calD^m}}_{I \times J \subset \Omega} |\langle b, h_I \otimes h_J \rangle|^2 \langle w \rangle_{I \times J}^{-1} \Big)^{1/2} \\
 &= \sup \Big\{ \Big| \mathop{\sum_{I \in \calD^n, J \in \calD^m}}_{I \times J \subset \Omega} \langle b, h_I \otimes h_J\rangle a_{I,J} \Big| \colon \mathop{\sum_{I \in \calD^n, J \in \calD^m}}_{I \times J \subset \Omega} |a_{I,J}|^2 \langle w \rangle_{I \times J} = 1\Big\}.
\end{align*}
Given such $a = (a_{I,J})$ define $f_a = \sum_{\substack{I, J \\ I \times J \subset \Omega}} a_{I,J} h_I \otimes h_J$. Then we have
using \eqref{eq:bfest} and H\"older's inequality that
\begin{align*}
\Big| &\mathop{\sum_{I \in \calD^n, J \in \calD^m}}_{I \times J \subset \Omega} \langle b, h_I \otimes h_J\rangle a_{I,J} \Big| = |\langle b, f_a\rangle| \\
&\lesssim_{[w]_{A_2}} \|b\|_{\bmo_{\calD^n, \calD^m}(w)}\Big\| \Big( \mathop{\sum_{I \in \calD^n, J \in \calD^m}}_{I \times J \subset \Omega} |a_{I,J}|^2 \frac{1_I \otimes 1_J}{|I||J|} \Big)^{1/2} \Big\|_{L^1(w)} \\
&\le \|b\|_{\bmo_{\calD^n, \calD^m}(w)} w(\Omega)^{1/2} \Big( \mathop{\sum_{I \in \calD^n, J \in \calD^m}}_{I \times J \subset \Omega} |a_{I,J}|^2 \langle w \rangle_{I \times J} \Big)^{1/2} = \|b\|_{\bmo_{\calD^n, \calD^m}(w)}w(\Omega)^{1/2}.
\end{align*}
We are done. 
\end{proof}
\begin{rem}
Notice that actually the proof works just by assuming that $w$ is uniformly in $A_{\infty}(\R^n)$ and $A_{\infty}(\R^m)$, and that we
also do not need the full strength of the little BMO assumption: we can e.g. use $\esssup_{x_1 \in \R^n} \, \|b(x_1, \cdot)\|_{\BMO_{\calD^m}(w(x_1, \cdot))}$
instead of the little BMO norm.
\end{rem}

\section{Boundedness of the strong maximal function}\label{app2}
We give a proof of the following variant of a result of Fefferman \cite{Fe3}. The proof is quite clear in the dyadic setting.
We note that we get a polynomial dependence on $[\lambda]_{A_p}$, while in Barron--Pipher \cite{BP} there seemed to be some
exponential dependence.
\begin{prop}
Let $p \in (1,\infty)$ and $\lambda \in A_p(\R^n \times \R^m)$. Then for $s \in (1,\infty)$ we have
$$
\| M_{\calD^n, \calD^m, \lambda} f \|_{L^s(\lambda)} \lesssim [\lambda]_{A_p}^{1+1/s} \|f\|_{L^s(\lambda)}.
$$
\end{prop}
\begin{proof}
Write $\calD = \calD^n \times \calD^m$.
By interpolation it is enough to prove $\| M_{\calD, \lambda} f \|_{L^{s,\infty}(\lambda)} \lesssim [\lambda]_{A_p}^{1+1/s} \|f\|_{L^s(\lambda)}$.
Fix $f$ and $\alpha > 0$, and set $\Omega = \Omega(\alpha) = \{ M_{\calD, \lambda} f > \alpha\}$. Write $\Omega = \bigcup_{j=1}^{\infty} R_j$ for some
rectangles $R_j \in \calD$
with $\langle |f| \rangle_{R_j}^{\lambda} > \alpha$. It suffices to fix $N$ and prove
\begin{equation}\label{eq:goalSM}
\alpha \lambda\Big(\bigcup_{j=1}^N R_j\Big)^{1/s} \lesssim [\lambda]_{A_p}^{1+1/s}\|f\|_{L^s(\lambda)}.
\end{equation}

Write $R_j = I_j \times J_j$, and reindex the cubes so that $\ell(J_{j+1}) \le \ell(J_j)$, $j = 1, \ldots, N-1$. We use Cordoba--Fefferman algorithm.
Let $s_1 = 1$, and suppose $s_1 < s_2 < \cdots < s_{l-1} < N$ have been chosen. Then $s_{l}$ is defined, if it exists, to be the smallest integer
$j \in (s_{l-1}, N]$ so that
$$
\Big|R_j \cap \bigcup_{i=1}^{l-1} R_{s_i}\Big| < \frac{|R_j|}{2}.
$$
Write $\calJ = \{1, \ldots, N\}$, $\calJ_s = \{s_i\}$ and $\calJ_s^c = \calJ \setminus \calJ_s$.
Notice that for all $j_0 \in \calJ$ we have
\begin{equation}\label{eq:cont}
R_{j_0} \cap \bigcup_{\substack{j \in \calJ_s \\ j < j_0}} R_{j} = \Big[ I_{j_0} \cap \bigcup_{\substack{j \in \calJ_s \\ j < j_0 \\ J_{j_0} \subset J_j}} I_{j} \Big] \times J_{j_0}.
\end{equation}

For $x_2 \in \R^m$ we set $I_j(x_2) = I_j$ if $x_2 \in J_j$, and $I_j(x_2) = \emptyset$ otherwise.
Let $j_0 \in \calJ_s^c$ be arbitrary. Then we have
$$
\Big| R_{j_0} \cap \bigcup_{\substack{j \in \calJ_s \\ j < j_0}} R_{j}\Big| \ge \frac{|R_{j_0}|}{2}.
$$
Using \eqref{eq:cont} we see that
$$
\Big| I_{j_0}(x_2) \cap \bigcup_{j \in \calJ_s} I_{j}(x_2)\Big| \ge \Big| I_{j_0}(x_2) \cap \bigcup_{\substack{j \in \calJ_s \\ j < j_0}} I_{j}(x_2)\Big| \ge \frac{|I_{j_0}(x_2)|}{2}.
$$
Using that for all cubes $I \subset \R^n$ and all subsets $E \subset I$ we have
\begin{equation}\label{eq:stan1}
\frac{\lambda(\cdot, x_2)(E)}{\lambda(\cdot, x_2)(I)} \ge [\lambda]_{A_p}^{-1} \Big( \frac{|E|}{|I|} \Big)^p
\end{equation}
we conclude that
$$
\lambda(\cdot, x_2)\Big(I_{j_0}(x_2) \cap \bigcup_{j \in \calJ_s} I_{j}(x_2)\Big) \ge  c_1 [\lambda]_{A_p}^{-1} \lambda(\cdot, x_2)(I_{j_0}(x_2)), \qquad c_1 := 2^{-p}.
$$
Since $j_0 \in \calJ_s^c$ was arbitrary we get for all $x_2 \in \R^m$ that
$$
\bigcup_{j \in \calJ_s^c} I_j(x_2) \subset \Big\{ M_{\calD^n, \lambda(\cdot, x_2)}\big( 1_{\bigcup_{j \in J_s} I_j(x_2)}\big) \ge c_1[\lambda]_{A_p}^{-1}\Big\}.
$$
Using that $M_{\calD^n, \lambda(\cdot, x_2)} \colon L^1(\lambda(\cdot, x_2)) \to  L^{1,\infty}(\lambda(\cdot, x_2))$ (even with constant $1$) we get
$$
\lambda(\cdot, x_2) \Big( \bigcup_{j \in \calJ} I_j(x_2) \Big) \lesssim [\lambda]_{A_p} \lambda(\cdot, x_2) \Big( \bigcup_{j \in \calJ_s} I_j(x_2) \Big),
$$
which, after integrating over $x_2 \in \R^m$, gives our first key inequality
\begin{equation}\label{eq:app1KEY1}
\lambda\Big( \bigcup_{j \in \calJ} R_j \Big) \lesssim [\lambda]_{A_p} \lambda \Big( \bigcup_{j \in \calJ_s} R_j \Big).
\end{equation}

Let now $j_0 \in \calJ_s$. Then by construction and using \eqref{eq:cont} we get for all $x_2$ that
$$
\Big| I_{j_0}(x_2) \cap \bigcup_{\substack{j \in \calJ_s \\ j < j_0}} I_{j}(x_2)\Big| \le \frac{|I_{j_0}(x_2)|}{2}.
$$
Applying \eqref{eq:stan1} to $E_{j_0}(x_2) := I_{j_0}(x_2) \setminus \bigcup_{\substack{j \in \calJ_s \\ j < j_0}} I_{j}(x_2)$ we have
$$
\lambda(\cdot, x_2)(E_j(x_2)) \ge c_1[\lambda]_{A_p}^{-1}\lambda(\cdot, x_2)(I_j(x_2)), \qquad j \in \calJ_s, \,\, x_2 \in \R^m.
$$
Dualising against $g$ with $\|g\|_{L^s(\lambda(\cdot, x_2))} \le 1$, using the above sparseness property and using
that $M_{\calD^n, \lambda(\cdot, x_2)} \colon L^s( \lambda(\cdot, x_2)) \to L^s(\lambda(\cdot, x_2))$ (with a norm independent of $\lambda$) we get
$$
\Big\| \sum_{j \in \calJ_s} 1_{I_j(x_2)} \Big\|_{L^{s'}(\lambda(\cdot, x_2))}^{s'} \lesssim [\lambda]_{A_p}^{s'} \lambda(\cdot, x_2)\Big( \bigcup_{j \in \calJ_s} I_j(x_2)\Big).
$$
Integrating over $x_2 \in \R^m$ we get our second key inequality
\begin{equation}\label{eq:app1KEY2}
\Big\| \sum_{j \in \calJ_s} 1_{R_j} \Big\|_{L^{s'}(\lambda)} \lesssim [\lambda]_{A_p} \lambda\Big( \bigcup_{j \in \calJ_s} R_j\Big)^{1/s'}.
\end{equation}
Recalling that $\langle |f| \rangle_{R_j}^{\lambda} > \alpha$ and using \eqref{eq:app1KEY1} and \eqref{eq:app1KEY2} we get our claim \eqref{eq:goalSM}:
\begin{equation*}
\begin{split}
\alpha \lambda\Big(\bigcup_{j=1}^N R_j\Big)^{1/s} 
 &\lesssim [\lambda]_{A_p}^{1/s} \alpha\frac{\lambda \big( \bigcup_{j \in \calJ_s} R_j \big)}{\lambda \big( \bigcup_{j \in \calJ_s} R_j \big)^{1/s'}}
\le [\lambda]_{A_p}^{1/s} \frac{\sum_{j \in \calJ_s} \int_{R_j} |f|  \lambda}{\lambda \big( \bigcup_{j \in \calJ_s} R_j \big)^{1/s'}} \\
& \le [\lambda]_{A_p}^{1/s} \frac{\big\| \sum_{j \in \calJ_s} 1_{R_j} \big\|_{L^{s'}(\lambda)}}{{\lambda \big( \bigcup_{j \in \calJ_s} R_j \big)^{1/s'}}} \| f \|_{L^s(\lambda)}
\lesssim [\lambda]_{A_p}^{1+1/s}\| f \|_{L^s(\lambda)}.
\end{split}
\end{equation*}
\end{proof}

\end{appendix}

\subsection*{Funding}
The work of K.L. was supported by the Basque Government BERC
2018-2021 program; and Spanish Ministry of Economy and Competitiveness
MINECO through Juan de la Cierva - Formaci\'on 2015 [FJCI-2015-24547],
BCAM Severo Ochoa excellence accreditation [SEV-2013-0323]
and project [MTM2017-82160-C2-1-P] funded by \linebreak (AEI/FEDER, UE) and
acronym ``HAQMEC''.

The work of H.M. was supported by the Academy of Finland [294840, 306901]; and the University of Helsinki three-year research grant [75160010]. 

The work of E.V. was supported by the Academy of Finland [306901, 307333].

\subsection*{Acknowledgements}
We thank Prof. T. Hyt\"onen for his generous help with Appendix \ref{app1}. We also thank the referees for suggestions that helped to clarify the exposition.
Part of this research was conducted when K. Li was visiting University of Helsinki -- the hospitality of which is acknowledged.
H.M. and E.V. are members of the Finnish Centre of Excellence in Analysis and Dynamics Research.

\end{document}